\documentclass[leqno,twoside, 11pt]{amsart}
\makeatletter
\@namedef{subjclassname@2020}{\textup{2020} Mathematics Subject Classification}
\makeatother

\usepackage{amssymb, latexsym, esint}
\usepackage[utf8]{inputenc}
\usepackage{amsmath}
\usepackage{amsfonts}
\usepackage{amssymb}
\usepackage{amsthm}
\usepackage{times}
\usepackage[dvipsnames]{xcolor}
\usepackage{enumitem}
\usepackage{xcolor}
\newtheorem{teo}{Theorem}
\newtheorem{lema}[teo]{Lemma}

\newtheorem{defi}[teo]{Definition}
\newtheorem{corol}[teo]{Remark}

\newcommand{\R}{\mathbb{R}}

\newcommand{\di}{\displaystyle}
\def\Wc{W^{s,G}_{0}(\Omega)}
\def\W{W^{s,G}}
\def\Wd{W^{-s,\widetilde{G}}(\Omega)}
\def\Gm{\widetilde{G}}
\def\Bm{\widetilde{B}}
\def\fm{\widetilde{f}}
\def\Hc{\textbf{H}_{0}}
\newcommand{\Rn}{\mathbb{R}^{n}}
\newcommand{\Rnn}{\mathbb{R}^{2n}}
\newcommand{\N}{\mathbb{N}}
\def\t{{\tau_{\varepsilon}}}
\def\uarr{\overline{u}}
\def\uab{\underline{u}}
\def\O{\Omega}
\def\Oc{\overline{\Omega}}

\def\Fi{\varphi}

\def\glap{(-\bigtriangleup_{g})^{s}}
\def\Su{S(\uab, \uarr)}
\def\Mru{\overline{M}_{1}}
\def\Mrd{\overline{M}_{2}}
\def\Mrt{\overline{M}_{3}}
\def\Mri{\overline{M}_{i}}
\title{}
\author{}

\begin{document}
\title[$g$-fractional Laplacian]{Existence and Multiplicity of solutions for a Dirichlet problem  in fractional Orlicz-Sobolev  spaces}
	\author{Pablo Ochoa}

\address{Pablo Ochoa. Universidad Nacional de Cuyo. CONICET. Universidad J. A. Maza\\Parque Gral. San Mart\'in 5500\\
Mendoza, Argentina.}.
\email{pablo.ochoa@ingenieria.uncuyo.edu.ar}

\author{Anal\'ia Silva}

\address{Anal\'ia Silva. Instituto de Matem\'atica Aplicada San luis (IMASL),
Universidad Nacional de San Luis, CONICET \\Ejercito de los Andes
950, D5700HHW \\ San Luis,
Argentina\\analiasilva.weebly.com}
\email{acsilva@unsl.edu.ar}

\author{Maria Jos\'e Suarez Marziani}

\address{Maria Jos\'e Suarez Marziani. Instituto de Matem\'atica Aplicada San luis (IMASL),
Universidad Nacional de San Luis, CONICET \\Ejercito de los Andes
950, D5700HHW \\ San Luis,
Argentina} \email{mjsuarez@email.unsl.edu.ar}


	\maketitle
	
	\begin{abstract}
		In this paper, we first prove the existence of solutions to Dirichlet problems involving the fractional $g$-Laplacian operator and lower order terms by appealing to sub- and supersolution methods. Moreover, we also state the existence of extremal solutions. Afterwards, and under additional assumptions on the lower order structure, we establish by variational techniques the existence of multiple solutions: one positive, one negative and one with non-constant sign.
	\end{abstract}
	
	\section{Introduction}
Let us consider the following non local and non standard growth problem
\begin{equation}\label{1.1}
		\left\{
		\begin{array}{ccll}
			(-\bigtriangleup_{g})^{s} u = f(x,u) &~~& \text{in}~~\O,\\
			\,\,\,\,\,u = 0&~~&\text{in}~~\O^{c},
		\end{array}
		\right.
	\end{equation}
where $s\in(0,1)$, $\Omega$ is a smooth and bounded domain in $\Rn$, and $(-\bigtriangleup_{g})^{s}$  is the fractional $g$-Laplacian defined for sufficiently smooth functions $u$ as
\begin{equation*}
			\glap u(x) = 2 p.v \int_{\Rn} g \left( \dfrac{u(x)-u(y)}{|x-y|^{s}} \right)\dfrac{dy}{|x-y|^{n+s}},
		\end{equation*}
		where $p.v.$ stands for 'in principal value' and $G'=g$ is an $N$-function. We recall the definition of $N$-function and its properties in Section 2. For a complete guide of this setting we  mention the now classical book \cite{KR}.\\
The fractional Orlicz-Sobolev spaces and their relation with the g-Laplacian operator were discussed in \cite{BoS}. These spaces are the appropriate functional framework for $(-\bigtriangleup_{g})^{s}$. For the reader convenience, we include some basic  definitions and properties  in Section 2.\\
 Throughout the paper,  we consider that $f$ has a subcritical growth in the sense of the Orlicz-Sobolev embedding \cite{AC}. The main feature on the non linear term $f$ is that no oddness condition is imposed.\\
As far as we know, the existence and multiplicity  of solutions for equations involving the fractional $g$-Laplacian operator have been approached recently. Indeed, in \cite{BBX}, they provide existence of a non-negative solution for fractional $g$-Laplacian problems with a source $f(x, u)$ having a power-growth in $u$. Moreover, in \cite{BMO} the existence of a nodal solution, that is, a changing-sing solution, was given. Regarding multiple solutions for $\glap$, we quote the work \cite{BOT}, for the existence of two non-trivial solutions, and  \cite{BO}, where infinitely many solutions are obtained for a class of non-local Orlicz-Sobolev Schr\"{o}dinger equations. \\Inspired by \cite{Base}, we obtain an existence result using the sub-supersolution method. More precisely,
	we consider a sub-supersolution pair $(\uab, \uarr)$ and we define the set \textit{S}$(\uab, \uarr)$ as follows
	\begin{equation*}
	\text{\textit{S}} (\uab, \uarr) = \left\lbrace u \in \Wc \colon \text{$u$ is a solution of \eqref{1.1}}, ~\uab \leq u \leq \uarr \right\rbrace.
	\end{equation*}
One of the main result of the article  states that  if $(\uab, \uarr)$ is a sub-supersolution pair of \eqref{1.1},  then there exists $u \in$ \textit{S} $(\uab, \uarr)$, and hence \eqref{1.1} {admits a solution.

Finally, we show, under some growth conditions in the source, the existence of three different  nontrivial solutions for equation \eqref{1.1}. More precisely, these solutions are one positive, one negative and one with non-constant sign.  The method that we employ  was introduced in  \cite{St} and was useful in different context, see for example \cite{DPFBS,Silva,FB,Li, J06} for the  local case  and \cite{Csp} for the nonlocal context. This technique} consists in restricting the functional associated to \eqref{1.1} to three different manifolds constructed by sign  and  normalization conditions. Unlike \cite{St}, here we do not appeal to the Ljusternik-Schnirelman theory.  Instead, we use the  Ekeland variational principle to prove the existence of a critical point of each  restricted functional which turns out to be a critical point of the unrestricted functional and, consequently, is  also a weak solution of \eqref{1.1}.\\
The paper is organized as follows. In Section 2 we collect some basic facts of Orlicz-Sobolev spaces and the fractional $g$-Laplacian used throughout the manuscript. Section 3 is devoted to the study of the set \textit{S} $(\uab, \uarr)$. Finally, in section 4 we prove the multiplicity of weak solutions for  \eqref{1.1}.

\section{Preliminaries}
\subsection{N-functions}
In this section we introduce basic definitions and preliminary results related to Orlicz spaces. For more details see \cite{KR}. We start recalling the definition of an N-function.
	\begin{defi}\label{d2.1}
		A function $G \colon \R_+ \rightarrow \R_+$ is called an N-function if it admits the representation
		$$G(t)= \int _{0} ^{t} g(\tau) d\tau,$$
		where the function $g$ is right-continuous for $t \geq 0$,  positive for $t >0$, non-decreasing, and satisfies the conditions
		$$g(0)=0, \quad g(\infty)=\lim_{t \to \infty}g(t)=\infty.$$
		\end{defi}We extend $G$ to $\R$ as an even function. By \cite[Chapter 1, Sec. 5]{KR}, an N-function has also the following properties:
		\begin{enumerate}
			\item $G$ is  continuous, convex and  even.
	\item $G$ is super-linear at zero and at infinite, that is $$\di\lim_{x\rightarrow 0} \dfrac{G(x)}{x}=0$$and
	$$\di\lim_{x\rightarrow \infty} \dfrac{G(x)}{x}=\infty.$$
	
		\end{enumerate}Indeed, the above conditions serve as an equivalent definition of N-functions.
		
		An important property for N-functions is the following:
		\begin{defi}
		
			 We say that the N-function  G satisfies the $\bigtriangleup_{2}$ condition if  there exists $C > 2$ such that
			\begin{equation*}
				G(2x) \leq C G(x) \,\,\text{~~~for all~~} x \in \R_{+}.
			\end{equation*}
			\end{defi}
			According to \cite[Theorem 4.1, Chapter 1]{KR} a necessary and sufficient condition for an N-function to satisfy the $\bigtriangleup_{2}$ condition is that there is $p^{+} > 1$ such that
			\begin{equation}
		\frac{tg(t)}{G(t)} \leq p^{+}, ~~~~~\forall\, t>0.
	\end{equation}
	
	Associate to $G$ is  the N-function  complementary to it which is defined as follows:
		\begin{equation}\label{Gcomp}
			\widetilde{G} (t) := \sup \left\lbrace tw-G(w) \colon w>0 \right\rbrace .
		\end{equation}
	
The definition of the complementary function assures that the following Young-type inequality holds
	\begin{equation}\label{2.5}
		at \leq G(t)+\Gm(a) \text{  for every } a,t \geq 0.
	\end{equation}
	
	By \cite[Theorem 4.3,   Chapter 1]{KR}, a necessary and sufficient condition for the N-function $\Gm$ complementary to $G$ to satisfy the $\bigtriangleup_{2}$ condition is that there is $p^{-} > 1$ such that
			\begin{equation*}
	p^{-} \leq 	\frac{tg(t)}{G(t)}, ~~~~~\forall\, t>0.
	\end{equation*}

	From now on, we will assume that the N-function  $G(t)= \int _{0} ^{t} g(\tau) d\tau$  satisfies the following growth behavior:
	\begin{equation}\label{G1}
		1 < p^{-} \leq \frac{tg(t)}{G(t)} \leq p^{+} < \infty, ~~~~~\forall t>0.
	\end{equation}Hence, $G$ and $\Gm$ both satisfy the $\bigtriangleup_{2}$ condition. Observe that \eqref{G1} holds  for all $t \in \mathbb{R}, \, t \neq 0.$ Another useful condition (which implies \eqref{G1}) is the following:
	\begin{equation}\label{cG0}
p^{-}-1 \leq \dfrac{tg'(t)}{g(t)}\leq p^{+}-1, ~~~~~\forall t>0.
\end{equation}Again, \eqref{cG0} holds for any $t \in \mathbb{R},\, t \neq 0.$
	
Examples of functions $G$ verifying the above conditions are:
\begin{itemize}
\item $G(t)=|t|^p$, $p >1$;
\item $G(t)= |t|^p+|t|^q$, $p, q >1$;
\item $G(t)= t^{p}(|\log(t)|+1) $, $p > (3+\sqrt{5})/2$;
\item $G(t)=\int_0^t\left( p|s|^{p-2}s(|\log(s)|+1) + \frac{|s|^{p-2}s}{1+|s|} \right)\,ds$, where $p^+= p+1$ and $p^-=p-1$ (see Remark 1 in \cite{Oc}).
\end{itemize}
 As a final assumption, we will finally impose that $g'$ is non-decreasing. \normalcolor
\subsection{Orlicz spaces}
Now, we  introduce the definition of the  fractional Orlicz-Sobolev spaces. For more details on this functional setting, see \cite{BoS}.
	\begin{defi} Let $s \in (0, 1)$
		and  $\O \subseteq \Rn$ be an open set. We define
		\begin{equation*}
			L^{G} (\O) :=  \left\lbrace u \colon \O \rightarrow \R,\,u \,\,is~~measurable\,\,and\,\,  \rho_{G,\O}(u) < \infty \right\rbrace
		\end{equation*}and
		\begin{equation*}
			\W (\O) := \left\lbrace u \in L^{G}(\O) \colon \rho_{s,G,\O}(u)< \infty \right\rbrace,
		\end{equation*}
		where the modulars $\rho_{G,\O}$ and $\rho_{s,G, \O}$ are defined as
		\begin{equation*}
				\rho_{G,\O}(u) :=  \di\int_{\O} G \left( |u(x)| \right) dx,
		\end{equation*}
		\begin{equation*}
				\rho_{s,G,\O} (u) := \di\iint_{\O\times\O} G \left( \dfrac{|u(x)-u(y)|}{|x-y|^{s}} \right) d\mu, \\
		\end{equation*}
		with
		\begin{equation*}
				d\mu := \dfrac{dxdy}{|x-y|^{n}}.
		\end{equation*}The norm associated to these spaces is the Luxemburg type norm
		\begin{equation*}
			||u||_{s,G,\O} := ||u||_{G,\O} + [u]_{s,G, \O}
		\end{equation*}
		where
		\begin{equation*}
				||u||_{G,\O} := \di\inf \left\lbrace  \lambda > 0 \colon \rho_{G,\O}\left(\dfrac{u}{\lambda}\right) \leq 1 \right\rbrace
		\end{equation*}
	and
		\begin{equation*}		
				\left[ u \right] _{s,G, \O} := \di\inf \left\lbrace \lambda > 0 \colon \rho_{s,G, \O}\left(\dfrac{u}{\lambda}\right) \leq 1 \right\rbrace .
		\end{equation*}When $\O=\Rn$, we will omit the dependence on $\O$ in the above quantities.\end{defi}

For further reference, we state  the H\"older inequality.
\begin{lema} For $u \in L^{G}(\O)$ and $v\in L^{\Gm}(\O)$ there holds
		\begin{equation*}
			\int_{\O} |uv| dx \leq 2||u||_{G,\O} ||v||_{\Gm,\O}.
		\end{equation*}
	\end{lema}
The topological dual space of $\W (\O)$ is denoted by $W^{-s, \widetilde{G}} (\O) $. We recall that when the N-function $G$ satisfies the $\bigtriangleup_{2}$ condition, then $L^{G}(\Rn)$ and $W^{s,G}(\Rn)$ are reflexive, separable Banach spaces (see \cite[Chapter 8]{Embedding}, \cite[Proposition 2.11]{BoS} and \cite[Theorem 8.2]{KR}).\\

		Also, to include boundary values, we define the space
		\begin{equation*}
			\Wc := \left\lbrace u \in W^{s,G}(\Rn) \colon u=0~~\text{in}~~\O^{c} \right\rbrace,
		\end{equation*}and the space of test functions
$$E:=\overline{C_c^\infty(\Omega)} \subset W^{s, G}(\Rn)$$where the closure is taking with respect to the norm of $\|\cdot\|_{s, G}$. Observe that $E \subset \Wc$. 

	Next, we mention some useful lemma for N-functions, which allows comparison with power functions.

	\begin{lema}\cite[Lemma 2.5]{BaS}\label{norm and modular}
		Let $$\xi^{-}(t) = \min\left\lbrace t^{p^{-}} , t^{p^{+}}\right\rbrace$$and $$\xi^{+}(t) =
		\max\left\lbrace t^{p^{-}}, t^{p^{+}}\right\rbrace.$$Then,
		\begin{enumerate}
			\item $\xi^{-} (||u||_{G,\O}) \leq \rho_{G,\O} (u) \leq \xi^{+} (||u||_{G,\O})$, for $u \in L^{G}(\O)$.
			\item $\xi^{-} (\left[ u \right] _{s,G}) \leq \rho_{s,G} (u) \leq \xi^{+} (\left[ u \right] _{s,G})$, for $u \in W^{s,G}(\O)$.
		\end{enumerate}
	\end{lema}
Now, we introduce a notion of comparison between $N$- functions.
\begin{defi}
		Given two $N$-functions $A$ and $B$, we say that $B$ is essentially larger than $A$, denoted by
$A \ll B,$ if for any $c > 0$,
		\begin{equation*}
			\di\lim_{t \rightarrow \infty} \dfrac{A(ct)}{B(t)}=0.
		\end{equation*}
	\end{defi}
	
	In order for the Sobolev immersion theorem to hold, it is necessary that G verifies the following two conditions:
\begin{itemize}
\item[(G1)]
		$\displaystyle\int_{M}^{\infty} \left( \frac{t}{G(t)} \right) ^{\frac{s}{n-s}} dt = \infty$ for some $M$.
\item[(G2)] $\displaystyle\int_{0}^{\delta} \left( \frac{t}{G(t)} \right) ^{\frac{s}{n-s}} dt < \infty$
	for some $\delta > 0$.
\end{itemize}

Observe that when $G(t)=|t|^{p}$, then assumptions (G1) and (G2) are satisfied if $sp < n$. More examples can be seen in Remark 1 in \cite{Oc}.
	
	Given  $G$ satisfying (G1) and (G2) we define its Orlicz-Sobolev conjugate	as
	\begin{equation}\label{2.5A}
		G_{\frac{n}{s}} (t) = G \circ H^{-1}(t),
	\end{equation}
	where
	\begin{equation*}
		H(t) = \left( \int_{0}^{t} \left( \frac{\tau}{G(\tau)} \right) ^{\frac{s}{n-s}} d\tau \right) ^{\frac{n-s}{n}}.
	\end{equation*}
	Hence, we have the following embedding theorem, see \cite{AC}.
	\begin{teo}\label{embedding}
		Let $G$ be an N-function satisfying \eqref{G1},    (G1) and (G2), and let $G_{\frac{n}{s}}$ be defined in \eqref{2.5A}. Then the embedding $W^{s,G}_{0} (\O) \hookrightarrow L^{G_{\frac{n}{s}}}(\O)$ is continuous. Moreover, the N-function  $G_{\frac{n}{s}}$ is optimal in the class of Orlicz spaces.
		Finally, given  any  N-function $B$, the embedding $W^{s,G}_{0} (\O) \hookrightarrow L^{B}(\O)$ is compact if and only if $B \ll G_{\frac{n}{s}}$.
	\end{teo}

\subsection{Elementary inequalities}	
In this subsection we collect some elementary inequalities for  N-functions $G$  satisfying \eqref{G1}. We give the proofs of those which, up to our knowledge,  are new in the literature.

\begin{lema}\label{Lema2.2}\cite[Lemma 2.9]{BoS}
		Let $G$ be an N-function satisfying \eqref{G1} such that $g=G'$. Then
		\begin{equation*}
			\Gm (g(t)) \leq (p^{+}-1) G(t)
		\end{equation*}
		holds for any $t \geq 0$.
	\end{lema}

	\begin{lema}
		Let $a, b \in \R$. Then we have
		\begin{equation}\label{A2}
			G \left( |a_{+}-b_{+}| \right) \leq (a_{+}-b_{+}) g(a-b),
		\end{equation}
 where $a_+=\max\{a,0\}.$
	\end{lema}
	\begin{proof}
		Let $a,b \in \R$.
		If $a \geq 0$ and $b \leq 0$, then
		\begin{equation}\label{eqq 1}
			G(|a_{+}-b_{+}|)=G(a_{+}) \leq \frac{a_{+}}{p^{-}}g(a_{+}) \leq (a_{+}-b_{+})g(a-b),
		\end{equation}where the last inequality follows from the fact that $g=G'$ is non-decreasing.
		The case $b_{+}\geq 0$ and $a_{+}=0$ is similar.
		
		Now, let $a_{+},b_{+}>0$, i.e. $a,b>0$. Then, if $a-b=a_{+}-b_{+}\geq 0$, we obtain from \eqref{eqq 1}  that
		\begin{equation*}
			G(|a_{+}-b_{+}|)=G(a_{+}-b_{+})\leq \frac{1}{p^{-}}(a_{+}-b_{+})g(a_{+}-b_{+}) \leq (a_{+}-b_{+})g(a-b).
		\end{equation*}
		If $a-b=a_{+}-b_{+} < 0$, and recalling that g is odd, we deduce
		\begin{equation*}
				G(|a_{+}-b_{+}|)= G(b_{+}-a_{+}) \leq \frac{1}{p^{-}} (b_{+}-a_{+})g(b_{+}-a_{+}) \leq (a_{+}-b_{+})g(a-b).
		\end{equation*}
	\end{proof}

	\begin{lema}\label{ineq g derivada} Let $G$ be an $N$-function satisfying  whose derivative $g$ satisfies \eqref{cG0}.
Then, for all $a, b \in \mathbb{R}$ it holds
$$g'(a-b)(a_{+}-b_{+})(a_{+}-b_{+}) \leq (p^{+}-1)g(a-b)(a_{+}-b_{+}).$$
	\end{lema}
	\begin{proof} The lemma is true if $a, b < 0$. Suppose that $a, b \geq 0$. If $a\geq b$, then, by \eqref{cG0},
	$$g'(a-b)(a_{+}-b_{+})(a_{+}-b_{+}) =g'(a-b)(a-b)(a_{+}-b_{+})  \leq (p^{+}-1)g(a-b)(a_{+}-b_{+}).$$If $a < b$,
	$$g'(a-b)(a_{+}-b_{+})(a_{+}-b_{+})= g'(b-a)(b-a)(b_{+}-a_{+}) \leq (p^{+}-1)g(b-a)(b_{+}-a_{+}).$$On the other hand, if $a < 0$ and $b> 0$, we have
	$$g'(a-b)(a_{+}-b_{+})(a_{+}-b_{+})= g'(a-b)(-b)(-b_{+}) \leq g'(b-a)(b-a)(b_{+}-a_{+})$$and we conclude as before. The case $a > 0$ and $b< 0$ is treated similarly.
	\end{proof}

	\begin{lema}\label{Lema2.1}\cite[Lemma 2.1]{MSV}
		 Assume \eqref{cG0} and that $g'$ is non-decreasing.  There exists $C > 0$ such that
		\begin{equation}\label{2.1}
			g(b)-g(a) \geq C g(b-a),
		\end{equation}
	for all $b \geq a$.
	\end{lema}
\begin{corol}\label{negativo}
Observe that as $g$ is odd
$$
g(b)-g(a) \leq C g(b-a),\qquad\mbox{ for all }b \leq a.
$$
\end{corol}	
\begin{lema}\label{desigualdad}
 Assume \eqref{cG0} and that $g'$ is non-decreasing. Let $a, b, c$ and $d$ be real numbers so that $a-b=c-d$. Then, there exists a constant $C>0$ such that
$$
G(|a-b|)\leq C(g(c)-g(d))(a-b).
$$
\end{lema}
\begin{proof}
First, we consider the case $a>b$, using \eqref{G1} and \eqref{2.1}
\begin{align*}
G(|a-b|)&=G(a-b)\leq \frac{1}{p^-} g(a-b)(a-b)=\frac{1}{p^-}g(c-d)(a-b)\\
&\leq C (g(c)-g(d))(a-b).
\end{align*}
The case $b>a$ follows analogously appealing to Remark \ref{negativo}.
\end{proof}

%

	\subsection{The fractional g-Laplacian operator.}
In this section we recall  some elementary properties of the fractional $g$-Laplacian.
		This operator is well defined between $\W(\Rn)$ and its dual space $W^{-s, \widetilde{G}}(\Rn)$. In fact, in \cite[Theorem 6.12]{BoS} the following representation formula is provided
		\begin{equation*}
			\langle \glap u, v \rangle = \iint _{\Rnn} g \left( \dfrac{u(x)-u(y)}{|x-y|^{s}} \right) \dfrac{v(x)-v(y)}{|x-y|^{s}} d\mu,
		\end{equation*}
		for any $v\in \W(\Rn)$.
	
	We give some properties of the operator $\glap $ in the next result. We  point out that the monotonicity of $\glap$ has been obtained in \cite[Lemma 3.4]{BBX} under a more restrictive assumption.
	\begin{lema}\label{Lemma2.1}
		$\glap \colon \Wc \rightarrow W^{-s, \widetilde{G}}(\Rn)$ is a monotone, continuous, and an $(S)_{+}$-operator. Moreover, it admits a continuous inverse.
	\end{lema}
	\begin{proof}The continuity of $\glap$ follows from H\"{o}lder inequality. Observe that to prove that it is monotone and an $(S)_{+}$-operator, it is enough by \cite[Example 27.2]{Z} to state that $\glap$ is uniformly monotone.
	Observe that for $u, v \in \Wc$, we obtained thanks to Lemma \ref{desigualdad} that
	\begin{equation}
	\begin{array}{rl}
	\rho_{s, G}(u-v)  &= \di\iint_{\Rnn}G\left(\dfrac{(u-v)(x)-(u-v)(y)}{|x-y|^{s}} \right)\,d\mu
	\\ &= \di\iint_{\Rnn}G\left(\bigg|\dfrac{(u-v)(x)-(u-v)(y)}{|x-y|^{s}} \bigg|\right)\,d\mu \\   
	&\leq C\left\langle \glap u- \glap v, u-v \right\rangle.
	\end{array} \end{equation}Thus, by Lemma \ref{norm and modular}, it follows that
	$$\left\langle \glap u-\glap v, u-v \right\rangle \geq C\min\left\lbrace \|u-v\|^{p^{+}-1}_{s, G}, \|u-v\|^{p^{-}-1}_{s, G}  \right\rbrace \|u-v\|_{s, G}.$$We conclude that $\glap$ is uniformly monotone. To prove that it admits a continuous inverse is enough to show that is coercive, hemicontinuous and uniformily monotone (see for instance Theorem 26.A in \cite{Z}). In fact,
$$
\frac{\left\langle \glap u,u\right\rangle}{[u]_{s,G}}\geq \frac{p^-\rho_{s,G}(u)}{[u]_{s,G}}\geq p^-\min \left\lbrace [u]_{s,G}^{p^--1},[u]_{s,G}^{p^+-1}\right\rbrace
$$so, taking limit as $[u]_{s,G}$ goes to infinity, we obtain that the fractional $g$-Laplacian is coercive.
Since the real function $t\to \left\langle \glap (u+tv),w\right\rangle$ is continuous in $[0,1]$ for any $u,v,w\in W^{s,G}(\Omega)$ we obtain that $\glap$ is hemicontinuous. This ends the proof.

	\end{proof}

\subsection{Notion of solution and regularity}
Let $X$ be an ordered Banach space, its nonnegative cone is denoted by $X_+$.
\begin{defi}
		Let $u \in W^{s, G}(\mathbb{R}^n)$. Then
		\begin{enumerate}[label=(\roman*)]
			\item $u$ is a supersolution of \eqref{1.1} if  
			\begin{equation*}
				\begin{array}{ccll}
					\langle \glap u, v \rangle \geq \di\int_{\O} f(x,u)v\, dx, \mbox{ for all } v \in E_{+},\\
				\end{array}
			\end{equation*}
			and $u \geq 0$ in $\O^{c}$.
			\item $u$ is a subsolution of \eqref{1.1} if 
			\begin{equation*}
				\begin{array}{ccll}
					\langle \glap u, v \rangle \leq \di\int_{\O} f(x,u)v\, dx, \mbox{ for all } v \in E_{+},\\
				\end{array}
			\end{equation*}
			and $u \leq 0$ in $\O^{c}$.
		\end{enumerate}
		Moreover , if $\uab$ is a subsolution, $\uarr$ is a supersolution, and $\uab \leq \uarr$ in $\O$, then the pair $(\uab, \uarr) \in \W (\Rn) \times \W (\Rn)$ is called a sub-supersolution pair of \eqref{1.1}.
	\end{defi}
	
	\begin{defi}\label{defi solution}We say that $u \in \Wc$ is a weak solution of \eqref{1.1} if 
	\begin{equation*}
				\begin{array}{ccll}
					\langle \glap u, v \rangle = \di\int_{\O} f(x,u)v\, dx, \mbox{ for all } v \in E.\\
				\end{array}
			\end{equation*}
	\end{defi}
\normalcolor
Finally, we  mention some results concerning the regularity of solutions to the $g$-Laplace equation. We refer the reader to \cite[Proposition 4.7]{BSV}, \cite[Theorem 1.1]{BSV} and \cite[Theorem 3]{BSV2}. \\
\begin{teo}\label{L2.5}Let $s \in (0, 1)$ and  let $G$ be an N-function satisfying  \eqref{G1} with $sp^{+} < n$. Moreover, assume that
\begin{equation}\label{cG}
\int_\Omega G(u(x))\,dx >0,
\end{equation}where $u \in W^{s, G}_0(\Omega)$ is a weak solution of the problem
\begin{equation*}
	\left\{
	\begin{array}{ccll}
		(-\Delta_g)^{s}u = b(u)\quad &~~& \text{in}~~\O,\\
		\,\,u = 0&~~&\text{in}~~\O^{c},
	\end{array}
	\right.
\end{equation*}where $b=B'$, with $B$ an N-function satisfying
$$\eta^{-} \leq  \dfrac{tb(t)}{B(t)}\leq \eta^{+},$$and $B \ll G_{\frac{n}{s}}$. Then, there exists a constant $C=C(s,n,  p^{\pm}, \eta^{\pm})>0$ such that
$$\|u\|_{L^{\infty}(\Omega)}\leq C.$$If $G$ additionally satisfies \eqref{cG0}, with $p^{-}>2$  and $g$ is a convex function in $(0, \infty)$, then there exists $\alpha \in (0, 1)$ such that
$$\|u\|_{C^{\alpha}(\overline{\Omega})} \leq C.$$
\end{teo}



	\section{Solutions in a sub-supersolution interval}
In this section we extend the results obtained in \cite{Base} to the Orlicz fractional setting. More precisely, we focus in proving  the existence of solutions for \eqref{1.1}. We assume that the function $f$ satisfies the following hypothesis:
	
	\
	
	$(\textbf{H}_{0})$~~~~~ $f \colon \O \times \R \rightarrow \R$ is a Carathéodory function such that for a.e $x \in \O$ and all $t \in \R$
	\begin{equation*}
		|f(x,t)| \leq c_{0} (1+|b(t)|),  ~~~~~~
	\end{equation*}
	where  $c_{0}>0$ and $b(t) = B'(t)$, with $B$ an N-function satisfying \eqref{G1} and $G \ll B \ll G_{\frac{n}{s}}$.
	
\
	
We point out  that this kind of hypothesis has been used before in the literature, see for example \cite{BaS}.

	
Now, we are in position to prove our first lemma.
	\begin{lema}\label{L3.1}
		Suppose that $f$ satisfies  $(\textbf{H}_{0})$. Let $u_{1}, u_{2} \in$ $W^{s,G}(\mathbb{R}^n)$:
		\begin{enumerate}[label=(\roman*)]
			\item\label{i} if $u_{1}, u_{2}$ are supersolutions of \eqref{1.1}, then  $u^*=\min \left\lbrace u_{1}, u_{2}\right\rbrace $ is also a supersolution;
			\item\label{ii} if $u_{1}, u_{2}$ are subsolutions of \eqref{1.1}, then  $u_*=\max \left\lbrace u_{1}, u_{2}\right\rbrace $ is a subsolution as well.
		\end{enumerate}
	\end{lema}
	\begin{proof}
		We  prove \ref{i}, the proof of \ref{ii} is analogous. We have for $i=1, 2$ and  for all $v \in E_{+}$ that
		\begin{equation}\label{3.1}
			\left\{
			\begin{array}{ccll}
			\di\iint_{\Rnn} g \left( \dfrac{u_{i}(x)-u_{i}(y)}{|x-y|^{s}} \right)\dfrac{v(x)-v(y)}{|x-y|^{s}} d \mu \geq \di\int_{\O} f(x,u_{i})v dx,\\
				u_{i} \geq 0 \mbox{ in }\O^{c}.
			\end{array}
		\right.
		\end{equation}
		Set $u^*=\min\left\lbrace u_{1}, u_{2}\right\rbrace \in W^{s,G}(\mathbb{R}^n)$, then $u^{*}\geq 0$ in $\O^{c}$.
		Moreover, we consider the following sets
			$$A_{1}= \left\lbrace x \in \Rn \colon u_{1} (x) < u_{2} (x) \right\rbrace, \quad A_{2}= A_{1}^{c}.$$
	Fix $\varepsilon > 0$, and define for all $t \in \R$ the truncation function
		\begin{equation*}
			\t (t) = \left\{
			\begin{array}{ccll}
				0 &~~&if~~t\leq 0,\\
				\dfrac{t}{\varepsilon}&~~&if~~0<t<\varepsilon,\\
				1&~~&if~~t \geq \varepsilon.
			\end{array}
			\right.
		\end{equation*}
	The mapping $\t \colon \R \rightarrow \R$ is Lipschitz continuous, non-decreasing,  $0 \leq \t(t) \leq 1$ for all $t \in \R$, and
	\begin{equation}\label{conv tau}
	\t (u_{2} - u_{1}) \rightarrow \mathcal{X} _{A_{1}},\qquad 1 -\t (u_{2} - u_{1}) \rightarrow \mathcal{X} _{A_{2}}
	\end{equation}a.e in $\Rn$  as $\varepsilon \rightarrow 0^{+}$. For simplicity, we denote
	$$\t:=\t (u_{2} - u_{1})$$and$$1-\t:=1 -\t (u_{2} - u_{1}).$$
 Let now $\Fi \in C^{\infty}_{c} (\O)_{+}$. Using $\t  \Fi$ and  $1 -\t \Fi$  as test functions	in \eqref{3.1}, we obtain
		\begin{equation} \label{3.2}
			\begin{split}
				I &:=   \di\iint_{\Rnn} g \left( \dfrac{u_{1}(x)-u_{1}(y)}{|x-y|^{s}} \right)  \dfrac{\t(x)\varphi(x)-\t(y)\varphi(y)}{|x-y|^{s}} d \mu \\
				& \quad + \di\iint_{\Rnn} g \left( \dfrac{u_{2}(x)-u_{2}(y)}{|x-y|^{s}} \right)  \dfrac{ (1-\t)(x)\varphi(x)-(1-\t)(y)\varphi(y)}{|x-y|^{s}} d \mu \\
				&\geq \di\int_{\O} f(x,u_{1})\t\Fi dx + \int_{\O} f(x,u_{2})(1-\t)\Fi dx := II.
			\end{split}
		\end{equation}
	We get, for $A_{1}$ and $A_{2}$,
		\begin{equation*}
		\begin{array}{llr}
				&I=\di\iint_{A_{1}\times A_{1}} g \left( \dfrac{u_{1}(x)-u_{1}(y)}{|x-y|^{s}} \right)  \dfrac{\Fi(x)-\Fi(y)}{|x-y|^{s}}[\t(x)] d \mu &(A)\\
				&+ \di\iint_{A_{1}\times A_{1}} g \left( \dfrac{u_{1}(x)-u_{1}(y)}{|x-y|^{s}} \right)\dfrac{\Fi(y)}{|x-y|^{s}}[\t(x) - \t(y)] d \mu &(B)\\
				&+ \di\iint_{A_{1} \times A_{2}} g \left( \dfrac{u_{1}(x)-u_{1}(y)}{|x-y|^{s}} \right) \dfrac{\Fi(x)}{|x-y|^{s}}[\t(x) ] d \mu &(C)\\
				&- \di\iint_{A_{2} \times A_{1}} g \left( \dfrac{u_{1}(x)-u_{1}(y)}{|x-y|^{s}} \right) \dfrac{\Fi(y)}{|x-y|^{s}}[\t(y) ] d \mu &(D) \\
				&+ \di\iint_{A_{1} \times A_{1}} g \left( \dfrac{u_{2}(x)-u_{2}(y)}{|x-y|^{s}} \right) \dfrac{\Fi(x)-\Fi(y)}{|x-y|^{s}}[1-\t(x)] d \mu &(E)\\
				&- \di\iint_{A_{1} \times A_{1}} g \left( \dfrac{u_{2}(x)-u_{2}(y)}{|x-y|^{s}} \right) \dfrac{\Fi(y)}{|x-y|^{s}}[\t(x) - \t(y)] d \mu &(B)\\
				&+ \di\iint_{A_{1} \times A_{2}} g \left( \dfrac{u_{2}(x)-u_{2}(y)}{|x-y|^{s}} \right) \dfrac{\Fi(x)-\Fi(y)}{|x-y|^{s}}[1-\t(x)] d \mu &(F)\\
				&- \di\iint_{A_{1} \times A_{2}} g \left( \dfrac{u_{2}(x)-u_{2}(y)}{|x-y|^{s}} \right) \dfrac{\Fi(y)}{|x-y|^{s}}[\t(x) ] d \mu &(C)\\
				&+ \di\iint_{A_{2} \times A_{1}} g \left( \dfrac{u_{2}(x)-u_{2}(y)}{|x-y|^{s}} \right) \dfrac{\Fi(x)}{|x-y|^{s}}[\t(y) ] d \mu &(D) \\
				&+ \di\iint_{A_{2} \times A_{1}} g \left( \dfrac{u_{2}(x)-u_{2}(y)}{|x-y|^{s}} \right) \dfrac{\Fi(x)-\Fi(y)}{|x-y|^{s}}[1-\t(y)] d \mu &(G)\\
				&+ \di\iint_{A_{2} \times A_{2}} g \left( \dfrac{u_{2}(x)-u_{2}(y)}{|x-y|^{s}} \right) \dfrac{\Fi(x)-\Fi(y)}{|x-y|^{s}} d \mu &(H).
		\end{array}
		\end{equation*}
	Observe first that  $(E), (F), (G) \rightarrow 0$ as $\varepsilon \rightarrow 0^{+}$ by \eqref{conv tau}. Also, getting together the integrals with equal letters, we have
			\begin{align*}
				&I = \di\iint_{A_{1} \times A_{1}} g \left( \dfrac{u_{1}(x)-u_{1}(y)}{|x-y|^{s}} \right)  \dfrac{\Fi(x)-\Fi(y)}{|x-y|^{s}}[\t(x)] d \mu\\
				&+  \di\iint_{A_{1} \times A_{1}} \left[ g \left( \dfrac{u_{1}(x)-u_{1}(y)}{|x-y|^{s}} \right) - g \left( \dfrac{u_{2}(x)-u_{2}(y)}{|x-y|^{s}} \right)  \right]
				 \dfrac{\Fi(y)}{|x-y|^{s}}[\t(x) - \t(y)] d \mu\\
				&+ \di\iint_{A_{1} \times A_{2}} \left[ g \left( \dfrac{u_{1}(x)-u_{1}(y)}{|x-y|^{s}} \right)  \dfrac{\Fi(x)}{|x-y|^{s}} - g \left( \dfrac{u_{2}(x)-u_{2}(y)}{|x-y|^{s}} \right) \dfrac{\Fi(y)}{|x-y|^{s}}\right]  \t(x)  d \mu\\
				&+ \di\iint_{A_{2} \times A_{1}} \left[g \left( \dfrac{u_{2}(x)-u_{2}(y)}{|x-y|^{s}}\right) \dfrac{\Fi(x)}{|x-y|^{s}} - g \left( \dfrac{u_{1}(x)-u_{1}(y)}{|x-y|^{s}} \right)  \dfrac{\Fi(y)}{|x-y|^{s}}\right] \t(y)  d \mu\\
				&+ \di\iint_{A_{2} \times A_{2}} g \left( \dfrac{u_{2}(x)-u_{2}(y)}{|x-y|^{s}} \right) \dfrac{\Fi(x)-\Fi(y)}{|x-y|^{s}} d \mu\\ & + \textbf{o}(1)\\
				&=(A)+(B)+(C)+(D)+(H).
			\end{align*}
		
We have that  $(B)$ is negative, since $g$ is non-decreasing and  since for all $x, y \in A_{1}$
		\begin{equation*}
			u_{1}(x)-u_{1}(y) \geq u_{2}(x)-u_{2}(y)
		\end{equation*}if and only if
\begin{equation*} 
\t(y)=\t(u_{2} - u_{1}) (y) \geq \t(u_{2} - u_{1}) (x)=\t(x).
\end{equation*}
	
	On the other hand, for all $x \in A_{1}$, $y \in A_{2}$, there holds
		\begin{equation*}
			u_{1}(x)-u_{1}(y) \leq u_{1}(x)-u_{2}(y) \leq u_{2}(x)-u_{2}(y),
		\end{equation*}
	and thus, using again  that $g$ is non-decreasing, we get
		\begin{equation*}
			 g \left( \frac{u_{1}(x)-u_{1}(y)}{|x-y|^{s}} \right) \leq
			 g \left( \frac{u_{1}(x)-u_{2}(y)}{|x-y|^{s}} \right) \leq
			 g \left( \frac{u_{2}(x)-u_{2}(y)}{|x-y|^{s}} \right).
		\end{equation*}
	Then,  $(C)$ is bounded from above by
$$
	\di\iint_{A_{1} \times A_{2}} \left[ g \left( \dfrac{u_{1}(x)-u_{2}(y)}{|x-y|^{s}} \right)  \dfrac{\Fi(x)-\Fi(y)}{|x-y|^{s}}\right] \tau_\varepsilon(x) d \mu.
$$
	%
	Analogously, for all $x \in A_{2}$, $y \in A_{1}$  we can bound  $(D)$ by
	
	$$
		\di\iint_{A_{2} \times A_{1}} \left[ g \left( \dfrac{u_{2}(x)-u_{1}(y)}{|x-y|^{s}} \right)   \dfrac{\Fi(x)-\Fi(y)}{|x-y|^{s}}\right]\tau_\varepsilon(y)  d \mu.
		$$
	
	So we have
	\begin{equation*}
		\begin{array}{llr}
			I &\leq \di\iint_{A_{1} \times A_{1}} g \left( \dfrac{u_{1}(x)-u_{1}(y)}{|x-y|^{s}} \right)  \dfrac{\Fi(x)-\Fi(y)}{|x-y|^{s}} \tau_\varepsilon(x) d \mu\\
			&+\di\iint_{A_{1} \times A_{2}} g \left( \dfrac{u_{1}(x)-u_{2}(y)}{|x-y|^{s}} \right)  \dfrac{\Fi(x)-\Fi(y)}{|x-y|^{s}} \tau_\varepsilon(x) d \mu\\
			&+\di\iint_{A_{2} \times A_{1}}  g \left( \dfrac{u_{2}(x)-u_{1}(y)}{|x-y|^{s}} \right)  \dfrac{\Fi(x)-\Fi(y)}{|x-y|^{s}} \tau_\varepsilon(y) d \mu \\
			&+ \di\iint_{A_{2} \times A_{2}} g \left( \dfrac{u_{2}(x)-u_{2}(y)}{|x-y|^{s}} \right)  \dfrac{\Fi(x)-\Fi(y)}{|x-y|^{s}} d \mu\,+ \textbf{o}(1)\\
			& \to   \di\iint_{\Rnn} g \left( \dfrac{u^*(x)-u^*(y)}{|x-y|^{s}} \right)  \dfrac{\Fi(x)-\Fi(y)}{|x-y|^{s}} d \mu,\quad \text{as $\varepsilon \to 0$}.
		\end{array}
	\end{equation*}
	
	Now in order to estimate II, we use the Dominated Convergence Theorem. Observe that the mayorant is obtained thanks to $(\textbf{H}_{0})$ and the definition of $\t$. In fact,
	\begin{equation*}
		|f(\cdot,u_{1})\t\Fi| \leq c_{0} (1+|b(u_{1})|)\Fi,
	\end{equation*}
	\begin{equation*}
		|f(\cdot,u_{2})(1-\t)\Fi| \leq c_{0} (1+|b(u_{2})|)\Fi.
	\end{equation*}
	So,  passing to the limit as $\varepsilon \rightarrow 0^{+}$,
	\begin{equation} \label{3.4}
		\begin{array}{ll}
			II  \to  \di\int_{\O} f(x,u^{*}) \Fi dx.
		\end{array}
	\end{equation}
	Hence, taking  $\varepsilon \rightarrow 0^{+}$ in \eqref{3.2}, we get for all $\Fi \in C^{\infty}_{c}(\O)_{+}$ that
	\begin{equation*}
		\iint_{\Rnn} g \left( \frac{u^{*}(x)-u^{*}(y)}{|x-y|^{s}} \right) \frac{\Fi(x)-\Fi(y)}{|x-y|^{s}} d \mu \geq \int_{\O} f(x,u^{*}) \Fi dx.
	\end{equation*}
	Finally by the definition of $E$, we obtain that the previous inequality holds for all test functions in $E_{+}$. Thus,  $u^{*}$ is a supersolution of \eqref{1.1}, as we want to prove.
	\end{proof}

Our next lemma proves that \textit{S}$(\uab, \uarr)$ is not empty.
	\begin{lema}\label{L3.2}
		Let $(\uab, \uarr)$ be a sub-supersolution pair of \eqref{1.1} with $f$ such that verifies $\textbf{H}_{0}$. Then, there exists $u \in$ \textit{S} $(\uab, \uarr)$.
	\end{lema}
	\begin{proof}
		For simplicity, we let $A= \glap \colon \Wc \rightarrow \Wd$.  By Lemma \ref{Lemma2.1}, $A$ is monotone and hemicontinuous. Thanks to \cite[Lemma 2.98 $(i)$]{8} we know that also  $A$ is pseudomonotone.
	
	We consider a truncation of $f$ defined as
	\begin{equation*}
		\fm (x,t) = \left\{
		\begin{array}{ccll}
			f(x,\uab(x)) &~~& if~~t\leq \uab(x),\\
			f(x,t) &~~& if~~\uab(x) < t < \uab(x),\\
			f(x,\uarr(x)) &~~& if~~t\geq \uarr(x),
		\end{array}
		\right.
	\end{equation*}
	for all $(x,t) \in \O \times \R$. Observe that$\fm$ does not satisfy $\Hc$ in general, however $\fm \colon \O \times \R \to \R$ is a Carathéodory function which  satisfies the following bound for a.e. $x$ and all $t \in \R$,
	\begin{eqnarray}\label{3.5}
		|\fm(x,t)| \leq c_0 (1+|b(\uarr(x))|+|b(\uab(x))|) \in L^{\Bm}(\O).
	\end{eqnarray}

	We define $D \colon \Wc \to \Wd$ as
	\begin{equation*}
		\langle D(u), v \rangle = - \di\int_{\O} \fm (x,u)v~dx,
	\end{equation*}
	for all $u, v \in \Wc$. Observe that $D$  is well defined since by \eqref{3.5} and $G\ll B$, we obtain
	\begin{equation*}
		\begin{array}{ll}
			\left| - \di\int_{\O} \fm (x,u)v~dx \right|  &\leq \di\int_{\O} |\fm (x,u)| |v|~dx\\
			&\leq  2 ||\fm (x,u)||_{\Bm,\Omega} ||v||_{B,\Omega} \\
			&\leq C ||\fm (x,u)||_{\Bm,\Omega} ||v||_{s,G,\Omega} \\
			&< \infty.
		\end{array}
	\end{equation*}
	Moreover, $D$ is linear and continuous.
		
	We now  prove that $D$ is strongly continuous. Let $\{u_{n}\}_{n\in\N}$ be a sequence such that $u_{n} \rightharpoonup u$ in $\Wc$, and take any subsequence of $\{D(u_{n})\}_{n\in\N}$ that we still denote by $\{D(u_{n})\}_{n\in\N}$. Passing to a further subsequence if necessary, we have $u_{n} \to u$ in $L^{B}(\O)$, $u_{n}(x) \to u(x)$ and $|u_{n}(x)| \leq h(x)$ for a.e. $x \in \O$, for some $h \in L^{B} (\O)$. Therefore, for all $n \in \N$, by $\Hc$ and Lemma  \ref{Lema2.2} we have for a.e. $x \in \O$
			\begin{equation*}
				\begin{array}{ll}
					|\fm (x, u_{n}) - \fm (x,u)| &\leq |\fm (x, u_{n})| + |\fm (x,u)| \\
					&= 2c_{0} + c_{0}| b(\uarr_{n})|+ c_{0} |b(\uab_{n})| + c_{0} |b(\uarr)|+ c_{0} |b(\uab)| \in L^{\Bm}(\O).
				\end{array}
			\end{equation*}
	Observe that by continuity of $f(x, \cdot)$ we have $\fm (x,u_{n}) \to \fm(x,u)$ a.e.  Moreover, for all $v \in \Wc$,
	\begin{equation}\label{rhs 1}
		\begin{array}{ll}
			|\langle D(u_{n}) - D(u), v \rangle| &\leq \di\int_{\O} |\fm (x, u_{n}) - \fm (x,u)| |v| dx\\
			&\leq 2 ||\fm (x, u_{n}) - \fm (x,u)||_{\Bm,\Omega} ||v||_{B,\Omega}\\
			&\leq C ||\fm (x, u_{n}) - \fm (x,u)||_{\Bm,\Omega} ||v||_{s, G,\Omega}.\\
		\end{array}
	\end{equation}
	By  dominated convergence Theorem, the right-hand side of \eqref{rhs 1}  tends to 0 uniformly for any $v$ with $\|v\|_{s, G,\Omega}\leq 1$. Hence $D(u_{n}) \to D(u)$ in $\Wd$. Therefore, any subsequence of $\{D(u_{n})\}_{n\in\N}$ has a further subsequence that converges to $D(u)$. Hence, the whole sequence $\{D(u_{n})\}_{n\in\N}$ converges to $D(u)$ and $D$ is strongly continuous. Thanks to \cite[Lemma 2.98 $(ii)$]{8}, $D$ is pseudomonotone. Moreover,  $A+D$ is pseudomonotone by \cite[Proposition 27.7]{Z}.
	
	Now we  prove that $A+D$ sends bounded sets into bounded set. We already stated that $A$ is bounded. Regarding  $D$, observe that
	\begin{equation*}
		\begin{array}{ll}
			||D(u)||_{-s,\Gm} &= \di\sup_{||v||_{s,G,\Omega} \leq 1} \langle D(u), v \rangle\\
			&= \di\sup_{||v||_{s,G,\Omega} \leq 1} - \di\int_{\O} \fm (x,u) v dx
			 \\
			&\leq \di\sup_{||v||_{s,G,\Omega} \leq 1}  2||\fm (x,u)||_{\Bm,\Omega} ||v||_{B,\Omega}\\&\leq C\di\sup_{||v||_{s,G,\Omega} \leq 1}   ||\fm (x,u)||_{\Bm,\Omega} ||v||_{s,G,\Omega}
\\ & \leq C\|c_0(1+ |b(\overline{u})|+ |b(\underline{u})|)\|_{\Bm,\Omega}. \\
		\end{array}
	\end{equation*}Thus, $D$ sends bounded sets into bounded sets.
	
	It remains to check that $A+D$ is coercive. In fact, for all $u \in \Wc$ with $||u||_{s,G}\geq 1$ we have
		\begin{equation*}
			\begin{array}{ll}
				\dfrac{\langle A(u) + D(u), u \rangle}{||u||_{s,G}} &= \dfrac{\langle A(u), u \rangle}{||u||_{s,G}} + \dfrac{\langle D(u), u \rangle}{||u||_{s,G}}\\
				&\geq \dfrac{p^-\rho_{s,G}(u)}{||u||_{s,G}} - \dfrac{1}{||u||_{s,G}} \di\int_{\O} \fm (x,u)u~dx\\
				&\geq \dfrac{p^-\xi^{-} (\left[ u \right] _{s,G})}{||u||_{s,G}} - \dfrac{1}{||u||_{s,G}} \di\int_{\O} |\fm (x,u)| |u|~dx\\
				&\geq p^-\left[ u \right] _{s,G} ^{p^{-}-1} - \dfrac{2}{||u||_{s,G}}||\fm (x,u)||_{\Bm,\Omega}||u||_{B,\Omega}\\
				&\geq p^-\left[ u \right] _{s,G} ^{p^{-}-1} - \dfrac{C}{||u||_{s,G}}||\fm (x,u)||_{\Bm,\Omega}||u||_{s,G}\\
				&\geq p^-\left[ u \right] _{s,G} ^{p^{-}-1}-C.
			\end{array}
		\end{equation*}
		Observe that this goes to infinity when   $\left[ u \right] _{s,G} \to \infty$, as we want to prove.
		
		Appealing to \cite[Theorem 2.99]{8},  the equation
		\begin{equation}\label{3.6}
			A(u)+D(u) = 0 \text{~~~~~in $\Wd$}
		\end{equation}
		has a solution $u \in \Wc$.
 Next,  we prove that in $\O$
		\begin{equation}\label{3.7}
			\uab \leq u \leq \uarr.
		\end{equation}
		Observe that \eqref{3.7} holds in $\O^{c}$. Now we use $(u-\uarr)_{+} \in \Wc_{+}$ as a  test function  in   \eqref{3.6}. Then
		\begin{equation*}
				\langle A(u) + D(u),(u-\uarr)_{+}  \rangle = \langle A(u) ,(u -\uarr)_{+}  \rangle + \langle D(u),(u- \uarr)_{+}  \rangle = 0.
		\end{equation*}Thus,
		\begin{equation*}
			\langle \glap u ,(u- \uarr)_{+}  \rangle = \di\int_{\O} \fm (x,u)(u -\uarr)_{+}~dx.
		\end{equation*}
		Observe now that in the support of $(u -\uarr)_{+}$ we have  $\fm (x,u) = f(x,\uarr)$. Hence,
		\begin{equation*}
			\begin{array}{ll}
			&\langle \glap u ,(u- \uarr)_{+}  \rangle = \di\int_{\O} f (x,\uarr)(u-\uarr)_{+}~dx\\
			&\leq  \di\iint_{\Rnn} g \left( \dfrac{\uarr(x)-\uarr(y)}{|x-y|^{s}} \right) \dfrac{(u-\uarr)_{+}(x)-(u-\uarr)_{+}(y)}{|x-y|^{s}} d\mu\\
			&= \langle \glap \uarr,(u-\uarr)_{+}  \rangle,
			\end{array}
		\end{equation*}
		where we also used that $\uarr$ is a supersolution of \eqref{1.1}, so
		\begin{equation}\label{ine lap}
			\langle \glap u - \glap \uarr ,(u- \uarr)_{+} \rangle \leq 0.
		\end{equation}
	
	In the next calculations, we first  use \eqref{A2} with $$a= \dfrac{(u-\uarr)(x)}{|x-y|^{s}}$$and$$b= \dfrac{(u-\uarr)(y)}{|x-y|^{s}}.$$Also,  by  \eqref{2.1}  and the fact that $u(x)-u(y)\leq \overline{u}(x)-\overline{u}(y)$ if and only if\\
	$(u-\uarr)_{+} (x) \leq (u-\uarr)_{+} (y)$, we have
		\begin{equation*}
			\begin{array}{ll}
				&\rho_{s,G} \left( (u- \uarr)_{+} \right) = \di\iint_{\Rnn} G \left( \dfrac{|(u-\uarr)_{+} (x) - (u-\uarr)_{+} (y)|}{|x-y|^{s}} \right) d \mu \\
				&\leq \di\iint_{\Rnn} g \left( \dfrac{(u-\uarr) (x) - (u-\uarr) (y)}{|x-y|^{s}} \right) \dfrac{(u-\uarr)_{+} (x) - (u-\uarr)_{+} (y)}{|x-y|^{s}} d \mu \\
				&\leq \dfrac{1}{C}  \di\iint_{\Rnn} \left[ g \left( \dfrac{u(x) - u(y)}{|x-y|^{s}} \right) - g \left( \dfrac{\uarr(x) - \uarr (y)}{|x-y|^{s}} \right) \right] \dfrac{(u-\uarr)_{+} (x) - (u-\uarr)_{+} (y)}{|x-y|^{s}} d \mu  \\
				&= \dfrac{1}{C} \di\iint_{\Rnn} g \left( \dfrac{u(x) - u(y)}{|x-y|^{s}} \right) \dfrac{(u-\uarr)_{+} (x) - (u-\uarr)_{+} (y)}{|x-y|^{s}} d\mu \\
				&-  \dfrac{1}{C} \di\iint_{\Rnn} g \left( \dfrac{\uarr(x) - \uarr (y)}{|x-y|^{s}} \right) \dfrac{(u-\uarr)_{+} (x) - (u-\uarr)_{+} (y)}{|x-y|^{s}} d \mu  \\
				&= C' \left[ \langle \glap u - \glap \uarr,(u-\uarr)_{+}  \rangle \right].
			\end{array}
		\end{equation*}
	Consequently, by \eqref{ine lap}, it follows $(u-\uarr)_{+}=0$. Analogously we prove $u \geq \uab$, so the inequality \eqref{3.7} holds.
	
	Finally, using \eqref{3.7} in \eqref{3.6} we see that $u \in \Wc$ solves \eqref{1.1} in the sense of Definition \ref{defi solution} (since $E \subset W^{s,G}(\mathbb{R}^n)$) and $u \in \Su$ as we want to prove.
	\end{proof}

We end the section giving a result related to extremal elements in $\Su$. For that, we will need two preliminary results. The proofs below are similar to those from \cite{Base}. However, we write them for convenience of the reader.

	We  recall the following definition.
	\begin{defi} We say that a partially ordered set $(S, \leq)$ is \textsl{downward directed} (resp. \textsl{upward directed}) if for all $u_{1}, u_{2} \in S$ there exists $u_{3} \in S$ such that $u_{3} \leq u_{1}, u_{2}$ (resp. $u_{3} \geq u_{1}, u_{2}$).
	
Also, we say that $S$ is \textsl{directed} if it is both downward and upward directed.
	
\end{defi}
	\begin{lema}\label{L3.3}
		Let $(\uab, \uarr)$ be a sub-supersolution pair of \eqref{1.1} with $f$ such that verifies $\textbf{H}_{0}$. Then, $\Su$ is directed.
	\end{lema}
	\begin{proof}
		We prove that $\Su$ is downward directed. Let $u_{1}, u_{2} \in \Su$, then in particular $u_{1}, u_{2}$ are supersolutions of \eqref{1.1}. Set $u^*= \min \left\lbrace u_{1}, u_{2} \right\rbrace  \in \Wc$, then by Lemma \ref{L3.1} $u^*$ is a supersolution of \eqref{1.1} and $\uab \leq u^*$. By Lemma \ref{L3.2} there exists $u_{3} \in S(\uab, u^*)$, in particular $u_{3} \in \Su$ and $u_{3} \leq u^*$.

		In the same way we see that $\Su$ is upward directed. And then $\Su$ is directed.
	\end{proof}
	
	\begin{lema}\label{L3.4}
	Let $(\uab, \uarr)$ be a sub-supersolution pair of \eqref{1.1} with $f$ such that verifies $\Hc$. Then $\Su$ is compact in $\Wc$.
	\end{lema}
	\begin{proof}
		Let $\{u_{n}\}_{n\in\N}$ be a sequence in $\Su$, then for all $n \in \N$ and $v\in \Wc$
	\begin{equation}\label{3.8}
		\langle \glap u_{n} ,v  \rangle = \di\int_{\O} f (x,u_{n})v~dx
	\end{equation}and $\uab \leq u_{n} \leq \uarr$. Using $u_{n} \in \Wc$ as a test function in \eqref{3.8}, we have by $\Hc$
	\begin{equation*}
		\begin{array}{ll}
			\langle \glap u_{n} , u_{n}  \rangle 
			&\leq \di\int_{\O} |f (x,u_{n})| |u_{n}|~dx \\
			&\leq c_{0} \left[ \di\int_{\O} |u_{n}| dx + \di\int_{\O} |b(u_{n})| |u_{n}|~dx\right] \\	
	&\leq c_{0} \left[ \di\int_{\O} |u_{n}| dx + p_B^+\di\int_{\O} B(|u_{n}|)~dx\right] \\	
&\leq c_0\left[\|\uab\|_{1,\Omega}+\|\uarr\|_{1,\Omega}+p_B^+ \dfrac{C}{2}\rho_{B,\Omega}(\uab)+p_B^+ \dfrac{C}{2}\rho_{B,\O}(\uarr)\right]
			\leq C.
		\end{array}
	\end{equation*}
	So $\{u_{n}\}_{n\in\N}$ is bounded in $\Wc$. Passing to a subsequence, we have $u_{n} \rightharpoonup u$ in $\Wc$, $u_{n}(x) \rightarrow u(x)$ and $|u_{n}|\leq h$ for a.e $x \in \Omega$ and $n\in\N$, with $h \in L^{G}(\O)$. Moreover,
	\begin{equation*}
			|f (x,u_{n})(u_{n}-u)| 
			\leq 2 c_{0} \left( 1+|b(h)|\right) \left( |\uab|+|\uarr|\right) \in L^{1}(\O).
	\end{equation*}

	Now, using $u_{n}-u \in \W_0 (\O)$ as a test function in \eqref{3.8}, we get
	\begin{equation}\label{eqq3}
		\langle \glap u_{n} ,u_{n}-u  \rangle = \di\int_{\O} f (x,u_{n})(u_{n}-u)~dx.
	\end{equation}
	Using Dominated convergence Theorem  we  have   that \eqref{eqq3}  tends to $0$ as $n \rightarrow \infty$. By Lemma \ref{Lemma2.1} we have $u_{n} \rightarrow u$ in $\Wc$. Finally, taking limit in \eqref{3.8} we obtain that $u \in \Su$.
	\end{proof}

	As an application of the previous results, we  obtain proceeding exactly as in \cite[Theorem 3.5]{Base} that $\Su$ contains extremal elements with respect to the pointwise ordering. Also, observe that we  apply Theorem \ref{L2.5} to have  $C^{\alpha}(\Oc)$ regularity of the solutions.
\begin{teo}\label{T3.5}
		Assume that $f(u)=b(u)$, $b(0)=0$,  where $b=B'$, with $B$ an N-function  so that $G \ll B \ll G_{\frac{n}{s}}$ and it satisfies \eqref{G1}. Moreover, assume that $G$ fulfils \eqref{cG} and \eqref{cG0}, with $g=G'$ convex in $(0, \infty)$ and $p^{-} > 2$. Let  $(\uab,\uarr)$ be a sub-supersolution pair of \eqref{1.1}. Then $\Su$ contains a smallest and a biggest element.
	\end{teo}

\section{Multiple solutions}
	In this section, we follow the approach of \cite{St} to prove the existence of three non-trivial different solutions of  \eqref{1.1}, two of constant sign and one nodal.

	 Throughout this section, we assume that the N-function $G$ satisfies  \eqref{cG0}.
	
 We let $$F(x, u):=\int_0^{u}f(x, \tau)\,d\tau,$$where$f$ satisfies $(\mathbf{H}_0)$ and the following hypothesis:
	\begin{enumerate}
		\item[$(\mathbf{H}_1)$]$f(x, 0)= 0$ for a.e. $x$. Also,   $f$ is continuously differentiable in $u$.
		\item[$(\mathbf{H}_2)$] There are an $N$-function $B$ with $G\ll B \ll G_{\frac{n}{s}}$ and $p^+<p_B^-$, and  constants $c_i> 0$, for $i=1, 2, 3, 4$, with $$c_2>p^{+},\quad  c_3<\frac{1}{p^{+}-1},$$ such that for all $u \in L^{B}(\Omega)$ there holds
		$$c_1\rho_{B,\Omega}(u) \leq c_2\int_\Omega F(x, u)\,dx \leq\int_\Omega f(x, u)u\,dx \leq c_3\int_\Omega f_u(x, u)u^{2}\,dx \leq c_4\rho_{B,\Omega}(u).$$
	\end{enumerate}
	
	\begin{corol}
	Observe that $(\mathbf{H}_1)$ implies that the trivial function  $u \equiv 0$ is a solution of \eqref{1.1}.\end{corol}
\begin{corol}Assume that $f(u)=b(u)$, $b(0)=0$,  where $b=B'$, $B$ an N-function  so that $G \ll B \ll G_{\frac{n}{s}}$,
$$p_B^{-}-1 \leq \dfrac{tb'(t)}{b(t)} \leq p^{+}_B-1,$$and $p^{+} <p_B^-$. Then $f$ satisfies $(\mathbf{H}_0),(\mathbf{H}_1)$ and $(\mathbf{H}_2)$. Indeed, $(\mathbf{H}_0),(\mathbf{H}_1)$ are clearly satisfied and for $(\mathbf{H}_2)$ observe that
\begin{equation*}
\begin{split}
& p_B^-\rho_{B,\Omega}(u)  =p_B^- \int_\Omega B(u(x))\,dx \leq  \int_\Omega b(|u(x)|)|u(x)|\, dx  \\ & \qquad=  \int_\Omega b(u(x))u(x)\, dx  \leq \dfrac{1}{p^{-}_B-1}\int_\Omega b'(|u(x)|)u(x)^{2}\,dx \leq \dfrac{p_B^{+}(p_B^{+}-1)}{p^{-}_B-1}\rho_{B, \Omega}(u).
\end{split}
\end{equation*}Hence, taking $c_1=c_2=p_B^-$, $c_3 = 1/(p^{-}_B-1)$, and $c_4=p_B^{+}(p_B^{+}-1)/(p^{-}_B-1)$ we get $(\mathbf{H}_2)$.

More generally, suppose that $f(x, \cdot)$ is odd, $f(x, u)\geq 0$ for $u \geq 0$, and the following pointwise relations hold:
\begin{equation}\label{cond q in f}
0 < q^{-}-1 \leq \dfrac{u f_u(x, u)}{f(x, u)} \leq q^{+}-1, \quad p^{+}+1 < q^{-} \leq q^{+},
\end{equation}
and
\begin{equation}\label{assumpt pointwise2}
b(u)\leq q^{-}f(x, u),\,\, f_u(x, u) \leq c b'(u),
\end{equation}for some $c >0$ and  all $x$ and $u$. Then $(\mathbf{H}_2)$ holds.  Indeed, we first have by \eqref{assumpt pointwise2},

\begin{equation}\label{example point}
\begin{split}
\left(1-\dfrac{1}{q^{-}} \right)\rho_{B,\Omega}(u)& =\left(1-\dfrac{1}{q^{-}} \right)\int_\Omega \int_0^{|u(x)|} b(t)\,dt \,dx \\&  \leq (q^{-}-1)\int_\Omega F(x, |u(x)|)\,dx=(q^{-}-1)\int_\Omega F(x, u(x))\,dx, 
\end{split}
\end{equation}where we have used that $f$ is odd in the last equality. This proves the first integral inequality in $(\mathbf{H}_2)$ with $c_1= 1-\frac{1}{q^{-}}$ and $c_2=q^{-}-1$. Next,
\begin{equation*}
\begin{split}
(q^{-}-1)\int_\Omega F(x, |u(x)|)\,dx & = (q^{-}-1)\int_\Omega\int_0^{|u(x)|}f(x, t)\,dt\,dx\\ &  \leq \int_\Omega\int_0^{|u(x)|}tf_t(x, t)\,dt\,dx \\& = \int_\Omega \left(tf(x, t)\bigg\vert_0^{|u(x)|}-\int_0^{|u(x)|}f(x, t)\right)\,dx\\ & \leq \int_\Omega |u|f(x, |u|)\,dx \\ & = \int_\Omega uf(x, u)\,dx.
\end{split}
\end{equation*}Finally, observe that by \eqref{cond q in f} and \eqref{assumpt pointwise2},
\begin{equation*}
\int_\Omega uf(x, u)\,dx \leq \dfrac{1}{q^{-}-1}\int_\Omega u^{2}f_u(x, u)\,dx \leq \dfrac{c p_B^{+}(p_B^{+}-1)}{q^{-}-1}\rho_{B, \Omega}(u),
\end{equation*}which ends the proof taking $c_3=1/(q^{-}-1)$ and $c_4=c p_B^{+}(p_B^{+}-1)/(q^{-}-1)$.


\end{corol}
	
	Observe that weak solutions of \eqref{1.1} are critical points of the functional \\$\Phi: W_0^{s, G}(\Omega)\to \mathbb{R}$ given by
	$$\Phi(u):= \iint_{\mathbb{R}^{2n}}G\left(\left|\dfrac{u(x)-u(y)}{|x-y|^{s}}\right| \right)\,d\mu - \mathcal{F}(u),$$ where
	$$\mathcal{F}(u):= \int_\Omega F(x, u)\,dx.$$
	Define the following subsets of $W_0^{s, G}(\Omega)$:
	$$M_1 := \left\lbrace u \in W_0^{s, G}(\Omega): \int_{\Omega} u_{+} dx > 0 \text{ and } \left\langle \glap u, u_{+}\right\rangle= \left\langle \mathcal{F}'(u), u_{+}\right\rangle \right\rbrace,$$
	$$M_2 := \left\lbrace u \in W_0^{s, G}(\Omega): \int_{\Omega} u_{-} dx > 0 \text{ and } \left\langle \glap u, u_{-}\right\rangle= \left\langle \mathcal{F}'(u), u_{-}\right\rangle \right\rbrace,$$ where $u_-=\max\{-u,0\}$ and
	$$M_3:=M_1\cap M_2.$$

Finally, we define
$$K_1:= \left\lbrace u \in M_1: u \geq 0 \right\rbrace,$$
$$K_2:= \left\lbrace u \in M_2: u \leq 0 \right\rbrace$$
and
$$K_3:=M_3.$$
	
	We start with the next lemma to show that these sets are not empty. We borrow some calculation from \cite{BMO} and write the proof in details for completeness.
\begin{lema}The sets $M_1$, $M_2$ and $M_3$ are non empty.
\end{lema}
\begin{proof}
Let
$$
\varphi_1(w)=\left\langle \glap w, w\right\rangle -\int_{\Omega}f(x,w)w\, dx.
$$
Take any $w_0 \in W_0^{s, G}(\Omega)_{+}$, with ${w_0}_{+} \neq 0$ in a set of positive measure. Observe that if $0<t<1$, by $(\mathbf{H}_2)$,
\begin{align*}
\varphi_1(tw_0)&=\left\langle \glap tw_0, tw_0\right\rangle-\int_{\Omega}f(x,tw_0)tw_0\, dx\\
&\geq \iint_{\mathbb{R}^{2n}} g\left(\dfrac{tw_0(x)-tw_0(y)}{|x-y|^{s}} \right)\dfrac{tw_0(x)-tw_0(y)}{|x-y|^{s}} \,d\mu -c_4\rho_{B,\Omega}(tw_0)\\& \geq
t^{p^+}\iint_{\mathbb{R}^{2n}} g\left(\dfrac{w_0(x)-w_0(y)}{|x-y|^{s}} \right)\dfrac{w_0(x)-w_0(y)}{|x-y|^{s}} \,d\mu -c_4\rho_{B,\Omega}(tw_0)\\
&\geq t^{p^+}\left\langle \glap w_0, w_0\right\rangle-c_4t^{p_B^-}\rho_{B,\Omega}(w_0)\\
&\geq t^{p^+}A_1-c_4t^{p_B^-} A_2,
\end{align*}
where $A_1=\left\langle \glap w_0, w_0\right\rangle \geq C \rho_{s, G}(w_0) \neq 0$ and $A_2=\rho_{B,\Omega}(w_0)$. As $p^+<p_B^-$ it follows that $\varphi_1(tw_0)\geq0$ for $t$ small enough. On the other hand, if $t\geq1$, analogously
$$
\varphi_1(tw_0)\leq t^{p^-}A_1-c_1t^{p_B^+}A_2.
$$
As $p^-<p^+<p_B^-<p_B^+$ it follows that $\varphi_1(tw_0)\leq0$ for $t$ big enough. Finally, by Bolzano's theorem there exists $t_0>0$ such that $t_0w_0\in M_1$. Analogously if $w_0 \leq 0$ with $w_0 < 0$ in a set of positive measure,  there exists  $t_1>0$ such that $t_1w_0\in M_2$.
Finally, we prove that $M_3 \neq \emptyset$. Let $w \in W_0^{s, G}(\Omega)$, with $w_{+}, w_{-}\neq 0$. For $s, t >0$, define
$$\varphi(t, s):=(\varphi_1(t, s), \varphi_2(t, s)),$$ where
$$\varphi_1(t, s):= \left\langle \glap (tw_{+}-sw_{-}), tw_{+}\right\rangle- \int_\Omega f(x,tw_{+}-sw_{-})tw_{+} \,dx$$and
$$\varphi_2(t, s):= \left\langle \glap (tw_{+}-sw_{-}), s(-w_{-})\right\rangle- \int_\Omega f(x,tw_{+}-sw_{-})s(-w_{-}) \,dx.$$Then, for $t \in [0, 1]$,
\begin{equation}
\begin{split}
\varphi_1(t, t) &= \iint_{\mathbb{R}^{2n}} g\left(\dfrac{(tw_+-tw_-)(x)-(tw_+-tw_-)(y)}{|x-y|^{s}} \right)\dfrac{tw_+(x)-tw_+(y)}{|x-y|^{s}} \,d\mu\\ & - \int_\Omega f(x,tw_{+}-tw_{-})tw_{+} \,dx \\ & \geq t^{p^{+}} \iint_{\mathbb{R}^{2n}} g\left(\dfrac{(w_+-w_-)(x)-(w_+-w_-)(y)}{|x-y|^{s}} \right)\dfrac{w_+(x)-w_+(y)}{|x-y|^{s}} \,d\mu\\ & - \int_\Omega f(x,tw_{+})tw_{+} \,dx  \\& \geq t^{p^{+}}A_1 -c_4t^{p_B^{-}}A_2,
\end{split}
\end{equation}where $A_1  \geq \rho_{s, G}(w_+)  \neq 0$ (by \ref{A2}) and $A_2 = \rho_{B, \Omega}(w_{+})$. Thus, for $t$ small we get $\varphi_1(t, t) >0$. Similarly, it holds $\varphi_2(t, t)> 0$. On the other hand, it also holds that for $t$ large that $\varphi_1(t, t) <0$ and $\varphi_2(t, t)< 0$. Next, we show that $\varphi_1(t, s)$ is increasing in $s$ for fixed $t$. Let $s_1 \leq s_2$ and consider as in \cite{BMO} the sets
$$D_1=\left\lbrace x \in \mathbb{R}^{n}: w(x)\geq 0\right\rbrace, \quad D_2:=D_1^{c}.$$Then,
\begin{equation*}
\begin{split}
&\varphi_1(t, s_2)-\varphi_1(t, s_1)  \\ & \quad = \left\langle \glap (tw_{+}-s_2w_{-}), tw_{+}\right\rangle- \int_\Omega f(x,tw_{+}-s_2w_{-})tw_{+} \,dx \\ &  \quad-  \left\langle \glap (tw_{+}-s_1w_{-}), tw_{+}\right\rangle+ \int_\Omega f(x,tw_{+}-s_1w_{-})tw_{+} \,dx \\ &  \qquad = \left\langle \glap (tw_{+}-s_2w_{-}), tw_{+}\right\rangle - \left\langle \glap (tw_{+}-s_1w_{-}), tw_{+}\right\rangle \\ & \quad = \int_{D_2}\int_{D_1}g\left(\dfrac{tw_+(x)+s_2w_-(y)}{|x-y|^{s}} \right) \dfrac{tw_+(x)}{|x-y|^{s}}\,d\mu   \\ &  \quad +  \int_{D_1}\int_{D_2}g\left(\dfrac{-s_2w_-(x)-tw_+(y)}{|x-y|^{s}} \right) \dfrac{-tw_+(y)}{|x-y|^{s}}\,d\mu  \\ &  \quad  -\int_{D_2}\int_{D_1}g\left(\dfrac{tw_+(x)+s_1w_-(y)}{|x-y|^{s}} \right) \dfrac{tw_+(x)}{|x-y|^{s}}\,d\mu \\ & \quad - \int_{D_1}\int_{D_2}g\left(\dfrac{-s_1w_-(x)-tw_+(y)}{|x-y|^{s}} \right) \dfrac{-tw_+(y)}{|x-y|^{s}}\,d\mu \\ & \quad =  \int_{D_2}\int_{D_1}\left[g\left(\dfrac{tw_+(x)+s_2w_-(y)}{|x-y|^{s}} \right)-g\left(\dfrac{tw_+(x)+s_1w_-(y)}{|x-y|^{s}} \right) \right]\dfrac{tw_+(x)}{|x-y|^{s}}\,d\mu \\ & \quad +
 \int_{D_1}\int_{D_2}\left[g\left(\dfrac{-s_2w_-(x)-tw_+(y)}{|x-y|^{s}} \right)-g\left(\dfrac{-s_1w_-(x)-tw_+(y)}{|x-y|^{s}} \right) \right]\dfrac{-tw_+(y)}{|x-y|^{s}}\,d\mu.
\end{split}\end{equation*}Hence, since $g$ is increasing, we get
$$\varphi_1(t, s_2)-\varphi_1(t, s_1)  \geq 0.$$Similarly, $$\varphi_2(t_1, s)-\varphi_2(t_2, s)  \geq 0,$$for $t_1 \leq t_2$ and fixed $s$. Thus, there are $r, R >0$, $r < R$ such that
$$\varphi_1(r, s) >0, \, \varphi_1(R, s)<0 \quad \text{ for all } s \in (r, R],$$
$$\varphi_2(t, r) >0, \, \varphi_2(t, R)<0 \quad \text{ for all } t \in (r, R],$$and consequently, there exists $t, s \in [r, R]$ such that $\varphi_1(t, s)=\varphi_2(t, s)=0$. This shows that $tw_+-sw_- \in M_3$.
\end{proof}

%
%
	\begin{lema}\label{Lema1PlanC}
		There exist constants $k_j > 0$, $j=1, 2, 3$, such that for every $u \in K_i$, $i=1, 2, 3$, there holds
		\begin{equation}\label{mult 1}
			\rho_{s, G}(u) \leq k_1\int_\Omega f(x, u)u\,dx \leq k_2\Phi(u) \leq k_3 \rho_{s, G}(u).
		\end{equation}
	\end{lema}
	\begin{proof}Let $u \in K_i$, for some $i=1, 2, 3.$ Then, by definition of the sets $K_i$,
	$$\left\langle \glap u, u\right\rangle = \int_\Omega f(x, u)u\,dx.$$Hence,
	$$\rho_{s, G}(u) \leq p^{-}\rho_{s, G}(u) \leq \left\langle \glap u, u\right\rangle = \int_\Omega f(x, u)u\,dx.$$ Moreover,   by $(\mathbf{H}_2)$, it follows that $\mathcal{F}(u)$ is non-negative. Then
		$$\Phi(u) = \rho_{s, G}(u) -\mathcal{F}(u) \leq \rho_{s, G}(u).$$ 	
		Consequently, by the definitions of the sets $K_i$ and  by $(\mathbf{H}_2)$, we have
		\begin{equation*}
			\begin{split}\Phi(u) & = \rho_{s, G}(u) -\mathcal{F}(u)  \\ & \geq \frac{1}{p^{+}}\left\langle \glap u, u\right\rangle -\int_\Omega F(x, u)\,dx \\ & \geq   \frac{1}{p^{+}}\int_\Omega f(x, u)u\,dx - \frac{1}{c_2} \int_\Omega f(x, u)u\,dx \\ & = \left( \frac{1}{p^{+}}-\frac{1}{c_2}\right) \int_\Omega f(x, u)u\,dx.
			\end{split}
		\end{equation*}Therefore, combining the above, we get \eqref{mult 1} choosing
		$$k_1=1,\, k_2=k_3=\left[\left( \frac{1}{p^{+}}-\frac{1}{c_2}\right) \right]^{-1}.$$
	\end{proof}
	\begin{lema}\label{Lema2PlanC} There exists a constant
		$ D > 0 $ such that $ [u]_{s,G}^{p^{+}} \geq D$, for all  $ u \in K_1$,
		$[u]_{s,G}^{p^{+}} \geq D$ for all $  u \in K_2$, and $[u_-]_{s,G}^{p^{+}}\, ,
		\,[u_+]_{s,G}^{p^{+}} \geq D$ for all $ u \in K_3$.
	\end{lema}
	\begin{proof}
		Let $u \in K_i$. Suppose, without loss of generality, that $[u_\pm]_{s,G}<1$.
		By definition of $K_i$ and hypothesis $(\mathbf{H}_2)$ we have
		\begin{equation*}
			[u_\pm]^{p+}_{s,G}\leq\rho_{s,G}(u_\pm)\,\leq k_1 \int_\Omega f(x,u)u_\pm \,dx  = k_1 \int_\Omega f(x,u_{\pm})u_\pm \,dx  \leq c \rho_{B,\Omega}(u_\pm).
		\end{equation*}
		Then using Sobolev inequality in Orlicz fractional spaces
		\begin{align*}
			[u_\pm]^{p+}_{s,G}&\leq c\max\{\|u_\pm\|_{B,\Omega}^{p_B^+},\|u_\pm\|_{B,\Omega}^{p_B^-}\}\\
			&\leq c\max\{[u_\pm]_{s,G}^{p_B^+},[u_\pm]_{s,G}^{p_B^-}\}=c[u_\pm]_{s,G}^{p_B^-}.
		\end{align*}
		As $p^+<p_B^-$, the proof is complete.\end{proof}

	\begin{lema}\label{Lema3PlanC}
		There exists $c>0$, $\delta>0$, $\Phi(u)\geq c[u]_{s,G}^{p^+}$ for every $u\in \Wc$ such that $[u]_{s,G}\leq \delta.$
	\end{lema}
	\begin{proof}
		Suppose that  $[u]_{s,G}\leq 1$, then by Sobolev embedding
		\begin{align*}
			\Phi(u)&=\rho_{s,G}(u)-\mathfrak{F}(u)\geq \rho_{s,G}(u)-c \rho_{B,\Omega}(u)\\
			&\geq \rho_{s,G}(u)-c\max\{\|u\|_{B,\Omega}^{p_B^+},\|u\|_{B,\Omega}^{p_B^-}\}\\
		&\geq [u]_{s,G}^{p^+}-c[u]_{s,G}^{p_B^-}\geq c[u]_{s,G}^{p^+}.\\
		\end{align*}
		If $[u]_{s,G}$ is small, as  $p^+<p_B^-$, the proof is complete.
	\end{proof}

	\begin{lema}\label{Lema4PlanC}
		$M_{i}$ is a $C^{1,1}$ submanifold of $\Wc$ of co-dimension 1 for $i=1,2$, and co-dimension  2 for i=3. 
	\end{lema}

	\begin{proof}
		We define
		\begin{equation*}
			\overline{M}_{1} := \left\lbrace u \in \Wc \colon \di\int_{\O} u_{+} dx > 0 \right\rbrace,
		\end{equation*}

		\begin{equation*}
			\overline{M}_{2} := \left\lbrace u \in \Wc \colon \di\int_{\O} u_{-} dx > 0 \right\rbrace
		\end{equation*}
	and
		\begin{equation*}
			\overline{M}_{3} := \overline{M}_{1} \cap \overline{M}_{2}.
		\end{equation*}

	It is enough to prove that $M_{i}$ is a regular sub-manifold of $\Wc$, thanks to the facts that $M_{i} \subset \overline{M}_{i}$ and the sets $\overline{M}_{i}$ are open.
	
	We will define a $C^{1,1}$-function $\varphi_{i} \colon \Mri \rightarrow \R^{d}$, where $d=1$  if $i=1, 2$ and $d=2$ if $i=3$
	and $M_{i}$ will be the inverse image of a regular value of $\varphi_{i}$.
	
	For $u \in \Mru$, 
	\begin{equation*}
		\varphi_{1}(u) = \left\langle \glap u, u_+\right\rangle  - \langle \mathcal{F}'(u), u_{+} \rangle.
	\end{equation*}
	
	Also, for $u \in \Mrd$, let
	\begin{equation*}
		\varphi_{2}(u) =  \left\langle \glap u, u_-\right\rangle - \langle \mathcal{F}'(u), u_{-} \rangle.
	\end{equation*}

	Finally, for $u \in \Mrt$,
	\begin{equation*}
		\varphi_{3}(u) = \left( \varphi_{1}(u),\varphi_{2}(u) \right).
	\end{equation*}

	We want to prove that 0 is a regular value for $\varphi_{i}$. In fact, for $u \in M_{1}$, we get from Lemma \ref{Lema1PlanC} and $(\mathbf{H}_2)$ that
	

\begin{equation}\label{cuent}
\begin{split}
\langle \varphi'_{1}(u), u_{+} \rangle & = \frac{d}{d \varepsilon}\left\langle \glap (u+\varepsilon u_+), (u+\varepsilon u_+)_+\right\rangle\vert_{\varepsilon=0}  \\ & - \di\int_{\O} f (x, u)u_{+} dx - \int_{\O} f_{u} (x, u)u_{+}^{2} dx.
\end{split}
\end{equation}Observe that
\begin{equation}\label{cuen}
\begin{split}
& \frac{d}{d \varepsilon}\left\langle \glap (u+\varepsilon u_+), (u+\varepsilon u_+)_+\right\rangle\vert_{\varepsilon=0} \\ & \quad = \iint_{\mathbb{R}^{2n}}g'\left(\dfrac{u(x)-u(y)}{|x-y|^{s}} \right) \left(\dfrac{u_+(x)-u_+(y)}{|x-y|^{s}}\right)^{2}d\mu \\ & +\iint_{\mathbb{R}^{2n}} g\left( \dfrac{u(x)-u(y)}{|x-y|^{s}}\right)\dfrac{u_+(x)-u_+(y)}{|x-y|^{s}}\,d\mu,
\end{split}
\end{equation}where in the last integrals we have used that
\begin{equation*}
\begin{split}
&\frac{d}{d\varepsilon}\left((u+\varepsilon u_+)_+(x)- (u+\varepsilon u_+)_+(y)\right)\vert_{\varepsilon=0}  = \begin{cases} 0 & \text{if }u(x), u(y) \leq 0, \\
-u(y)& \text{if }u(x)<0, u(y) >0, \\
u(x) & \text{if }u(x)>0, u(y) <0, \\
u(x)-u(y) & \text{if }u(x), u(y) \geq 0,\\
\end{cases}\\ & \qquad \qquad   =u_+(x)-u_+(y).
\end{split}
\end{equation*}Hence, by Lemma \ref{ineq g derivada}, we get  from  \eqref{cuen} that
$$ \frac{d}{d \varepsilon}\left\langle \glap (u+\varepsilon u_+), (u+\varepsilon u_+)_+\right\rangle\vert_{\varepsilon=0} \leq p^{+}\left\langle \glap u, u_+\right\rangle.$$Thus, we get from \eqref{cuent} that
	\begin{equation}\label{eqq 2}
		\begin{array}{ll}
			\langle \varphi'_{1}(u), u_{+} \rangle
			&\leq (p^{+} -1) \di\int_{\O} f (x, u)u_{+} dx - \int_{\O} f_{u} (x, u)u_{+}^{2} dx \\
			&\leq \left( p^{+} -1 -\dfrac{1}{c_{3}} \right)  \di\int_{\O} f (x, u)u_{+} dx,
		\end{array}
	\end{equation}
	where $p^{+}-1 - \dfrac{1}{c_{3}} < 0$ by $(\mathbf{H}_2)$.  Now, by Lemma \ref{Lema1PlanC}, we obtain
	\begin{equation*}
		\langle \varphi'_{1}(u), u_{+} \rangle \leq - C \rho_{s,G}(u_{+})
	\end{equation*}
	which is strictly negative by Lemma \ref{Lema2PlanC}. Then $\varphi'_{1}(u) \neq 0$ and we have $M_{1} = \varphi^{-1}_{1} (0)$ is a smooth sub-manifold of $\Wc$.
	
	In a similar way we can prove that  $M_{2}$ is also a smooth sub-manifold of $\Wc$.
	
We next consider $M_3$. For $u \in M_{3}$, we have \begin{equation*}
\begin{split}
&\left\langle \varphi'_1(u), u\right\rangle = \iint_{\mathbb{R}^{2n}}g'\left( \dfrac{u(x)-u(y)}{|x-y|^{s}}\right)\dfrac{u(x)-u(y)}{|x-y|^{s}}\dfrac{u_+(x)-u_+(y)}{|x-y|^{s}}\,d\mu \\ & \quad  +\iint_{\mathbb{R}^{2n}}g\left( \dfrac{u(x)-u(y)}{|x-y|^{s}}\right)\dfrac{u_+(x)-u_+(y)}{|x-y|^{s}}d\mu -\int_\O f(x, u)u_+dx - \int_\O f_u(x, u)u_+^{2}\,dx.
\end{split}
\end{equation*}Appealing to the fact that
$$sign(a-b)= sign (a_{+}-b_{+}),$$for all $a, b \in \R$ such that $a-b, a_+-b_+ \neq 0$, and using \eqref{cG0}, it follows that
\begin{equation*}
\begin{split}
& \iint_{\mathbb{R}^{2n}}g'\left( \dfrac{u(x)-u(y)}{|x-y|^{s}}\right)\dfrac{u(x)-u(y)}{|x-y|^{s}}\dfrac{u_+(x)-u_+(y)}{|x-y|^{s}}\,d\mu \\ & \qquad  +\iint_{\mathbb{R}^{2n}}g\left( \dfrac{u(x)-u(y)}{|x-y|^{s}}\right)\dfrac{u_+(x)-u_+(y)}{|x-y|^{s}}\,d\mu \\ & \qquad \leq p^{+}\iint_{\mathbb{R}^{2n}}g\left( \dfrac{u(x)-u(y)}{|x-y|^{s}}\right)\dfrac{u_+(x)-u_+(y)}{|x-y|^{s}}d\mu \\ & \qquad =p^{+}\left\langle \glap u, u_+ \right\rangle.
\end{split}
\end{equation*}Hence, recalling that $u \in M_1$ and proceeding as in \eqref{eqq 2}, we obtain that
$$\left\langle \varphi'_1(u), u\right\rangle < 0.$$Similarly, by the fact that
$$sign(a-b) \neq sign (a_{-}-b_{-}),$$for all $a, b \in \R$ such that $a-b, a_--b_- \neq 0$, using \eqref{cG0} and the choice of $c_3$ in ($\textbf{H}_2$), it follows that
$$\left\langle \varphi'_2(u), u\right\rangle \leq p^{-}\left\langle \glap u, u_+ \right\rangle - \di\int_{\O} f (x, u)u_{-} dx - \int_{\O} f_{u} (x, u)u_{-}^{2} dx < 0.$$ 

%
%
%

 Therefore $M_{3}$ is a smooth sub-manifold of $\Wc$.
\end{proof}
\begin{lema}\label{Lema4PlanC-completo}
The sets $K_{i}$ are complete. 
\end{lema}
\begin{proof}
We consider $K_{3}$, the proofs for $K_1$ and $K_2$ are similar. Let $\{u_{k}\}_{k\in\N}$ be a Cauchy sequence in $K_{3}$, then $u_{k} \rightarrow u$ in $\Wc$. In particular, $\glap u_k \to \glap u$ in $W^{-s, \Gm}(\Omega)$ and $u_{k_{+}} \rightharpoonup u_{+}$ in $\Wc$. So,
	\begin{equation}
	\begin{split}
	& \lim_{k \to \infty}\left\langle \glap u_k, u_{k_{+}}\right\rangle \\ & \quad = \lim_{k \to \infty} \left[\left\langle \glap u_k - \glap u, u_{k_{+}}\right\rangle + \left\langle \glap u, u_{k_{+}}\right\rangle \right] \\ & \quad = \left\langle \glap u, u_{+}\right\rangle.
	\end{split}
	\end{equation}Moreover, using the compact embedding $\Wc \hookrightarrow L^{B}(\O)$, we get $u_{k_{+}} \rightarrow u_{+}$ in $L^{B}(\O)$ and hence by $(\mathbf{H}_0)$, we obtain
	\begin{equation*}
		\di\int_{\O} f(x, u_{k})u_{k_{+}} dx \rightarrow \int_{\O} f(x, u)u_{+} dx.
	\end{equation*}Hence, recalling Lemma \ref{Lema2PlanC}, we  get $u_+ \in M_1$. Similarly, $u_- \in M_2$, and thus we obtain $u \in K_3$. This shows that $K_3$ is complete.
\end{proof}

	For the next result, we recall the definition of tangent space of Banach manifolds. 
\begin{defi}
Let $X$ be a Banach space and $M \subset X$ be a Banach manifold. The tangent space at $u$ of  $M$	is
$$
T_{u}M=\{v\colon\exists\alpha:(-1,1)
\to M\quad\alpha(0)=u\quad\mbox{ and }\quad\alpha^\prime(0)=v\}. $$
\end{defi}We recall that in Banach manifolds, $T_{u}M$ is a closed linear subspace of $X$.

	\begin{lema}\label{Lema4PlanC-des} For every $u \in M_{i}$ we obtain the direct decomposition
		\begin{equation*}
			T_{u}\Wc = T_{u} M_{i} \oplus span \{u_{+}, u_{-} \}.
		\end{equation*}
		Moreover, the projection onto the first component in this decomposition is uniformly continuous on bounded sets of $M_{i}$.
	\end{lema}
	\begin{proof}
We show the descomposition of $M_1$.
Let $v \in T_{u}\Wc$ and $v = v_{1} + v_{2}$, where $v_{2} = \alpha u_{+}$ and $v_{1} = v - v_{2}$. We are interesting in choosing $\alpha$ so that $v_{1} \in T_{u} M_{1}$,
	\begin{equation*}
		\langle \varphi'_{1}(u), v \rangle = \langle \varphi'_{1}(u), \alpha u_{+} \rangle + \langle \varphi'_{1}(u), v_{1} \rangle.
	\end{equation*}
	
	If we choose
	\begin{equation*}
		\alpha = \dfrac{\langle \varphi'_{1}(u), v \rangle}{\langle \varphi'_{1}(u), u_{+} \rangle},
	\end{equation*}then $\langle \varphi'_{1}(u), v_{1} \rangle = 0$.
	So, we obtain
	\begin{equation*}
		T_{u}\Wc  = T_{u} M_{1} \oplus  span \{ u_{+} \},
	\end{equation*}
	where $M_{1} = \left\lbrace u \colon \varphi_{1}(u)=0 \right\rbrace$ and $T_{u} M_{1} = \left\lbrace v \colon \langle \varphi'_{1}(u), v \rangle = 0 \right\rbrace$.
	Analogously,
	\begin{equation*}
			T_{u}\Wc  = T_{u} M_{2} \oplus  span \{ u_{-} \}
		\end{equation*} and
	\begin{equation*}
			T_{u}\Wc  = T_{u} M_{3} \oplus  span \{ u_{+}, u_{-} \}.
	\end{equation*}
	This finishes the proof.
	\end{proof}
	
	In the next lemma, we prove that the unrestricted  functional $\Phi$ satisfies the Palais-Smale condition, that is, if $\{u_{j}\}_{j \in \N} \subset \Wc$ is a sequence such that $\Phi(u_{j})$ is uniformly bounded and $\Phi' (u_{j}) \rightarrow 0$ strongly in $W^{-s, \widetilde{G}} (\O)$, then $\{u_{j}\}_{j \in \N}$ contains a strongly convergent subsequence. The arguments of the proof are standard. However, we write it for convenience of the reader.
	\begin{lema}\label{PS}
		The unrestricted functional $\Phi$ satisfies the Palais-Smale condition.
	\end{lema}
	\begin{proof}
		Let $\{u_{j}\}_{j \in \N} \subset \Wc$ be a Palais-Smale sequence, i.e $\Phi(u_{j})$ is uniformly bounded and $\Phi' (u_{j}) \rightarrow 0$ strongly in $W^{-s, \widetilde{G}} (\O)$. Since $\Phi(u_{j})$ is uniformly bounded, using Lemma \ref{Lema1PlanC}, $ \{u_{j}\}_{j \in \N}$ is bounded in $\Wc$.
		
		We define $\Phi'(u_{j}) := \psi _{j}$, then
		\begin{equation*}
			\langle \Phi'(u_{j}), z \rangle = \langle \psi_{j}, z \rangle, ~~for~all~z\in \Wc.
		\end{equation*}
		
		Also,
		\begin{equation*}
			\langle \Phi'(u_{j}), z \rangle = \di\iint_{\Rnn} g \left( \dfrac{u_{j}(x) - u_{j} (y)}{|x-y|^{s}}\right) \dfrac{z(x)-z(y)}{|x-y|^{s}}  d\mu - \di\int_{\O} f (x, u_{j})z dx.
		\end{equation*}
	
		Then, $v$ is a weak solution of problem
		\begin{equation*}
			\left\{
			\begin{array}{ccll}
				\qquad (-\bigtriangleup_{g})^{s} v = f(x,u_{j}) + \psi_{j} &~~& \text{in}~~\O,\\
				v= 0&~~&\text{in}~~\O^{c}.
			\end{array}
			\right.
		\end{equation*}Observe that by $(\textbf{H}_0)$, $f_{j} := f(x,u_{j}) + \psi_{j} \in W^{-s, \widetilde{G}} (\O)$.
	
		Now, we define the operator $T \colon W^{-s, \widetilde{G}} (\O) \rightarrow \Wc$, as $T(h) : = u$, where $u$ is a weak solution of problem
		\begin{equation*}
			\left\{
			\begin{array}{ccll}
				(-\bigtriangleup_{g})^{s} u =h &~~& \text{in}~~\O,\\
				\quad\quad\,\,\,\,\,\,\,u = 0&~~&\text{in}~~\O^{c}.
			\end{array}
			\right.
		\end{equation*}

By Lemma \ref{Lema2.1} we know that $(-\bigtriangleup_{g})^{s}$ is continuous and that admits a continuous	inverse on $W^{-s, \widetilde{G}} (\O)$.	It is enough to see that $f_{j} \rightarrow f(\cdot, u(\cdot))$ in $W^{-s, \widetilde{G}} (\O)$. As $\Phi' (u_{j}) \rightarrow 0$ strongly in $W^{-s, \widetilde{G}} (\O)$, we only need to check that $\{f(x,u_{j})\}_{j\in\N}$ has a subsequence that strongly converges in $W^{-s, \widetilde{G}} (\O)$, the proof of this fact can be found in \cite[Lemma 3.5]{BaS}. This finishes the proof.
		
	\end{proof}
	In the sequel, we will need the following  fact from \cite[Lemma 4.4]{MO23}: 
		\begin{equation}\label{ineq u pm}
		\Phi(u) >\Phi(u_{+})+ \Phi(-u_{-}),
		\end{equation}for all $u \in W^{s, G}(\mathbb{R}^n)$ with $u_{\pm}\neq 0$.

	The following result follows exactly as in \cite[Lemma 5]{J06}, using Lemma \ref{Lema4PlanC-des} together with \eqref{ineq u pm}, Lemma \ref{Lema2PlanC} with Poincar\'e inequality (see Proposition 3.2 in \cite{Sa}), the previous lemma and the fact that $M_i$ are complete manifolds.
	\begin{lema}\label{L28}
		The restricted functionals $\Phi \mid _{M_{i}}$ satisfy the Palais-Smale condition.
	\end{lema}

Moreover, we get as in \cite{St}	
	
	\begin{lema}\label{L29}
		Let $u \in M_{i}$ be a critical point of the restricted functional $\Phi \mid _{M_{i}}$. Then $u$ is also a critical point of the unrestricted functional $\Phi$ and hence a weak solution to the problem \eqref{1.1} where $f$ verifies $(\mathbf{H}_0)-(\mathbf{H}_2)$.
	\end{lema}

Finally, we give the proof of the existence of two constant sign solutions and a nodal solution to problem \eqref{1.1}. Here, we mainly follow \cite{Liu}.

	\begin{teo}
		Under assumptions $(\mathbf{H}_0)–(\mathbf{H}_2)$, there exists three different, nontrivial, weak solutions of problem \eqref{1.1}. Moreover,  we have that one is positive, one is negative and the other one has non-constant sign.
	\end{teo}

	\begin{proof}
	 From Lemma \ref{Lema3PlanC} and Poincar\'e inequality (Proposition 3.2 in \cite{Sa}), there exists a constant $0<c<1$ such that
	$$\Phi(u) \geq c[u]^{p^{+}}_{s, G}$$when $[u_{-}]_{s, G} < c$. By continuity of the projection $P(u)= u_{-}$, the set $U=\left\lbrace u \in M_1: [u_{-}]_{s, G} < c\right\rbrace$ is open, contains $K_1$ and its closure $\overline{U}$ is complete since $M_1$ is complete and $\overline{U}$ is closed. Since $\Phi$ is bounded from below in $\overline{U}$, we let
	$$m:= \inf_{\overline{U}}\Phi(u).$$ Take $u_j$ a minimizing sequence. By \eqref{ineq u pm}, we get that $u_{j, +}$ is also a minimizing sequence for $\Phi$ and moreover $u_{j, +}\in K_1$. For all $j$ large enough, take $\varepsilon_j>0$ such that
	$$\Phi(u_{j, +}) < m+\varepsilon_j, \quad \varepsilon < c^{p_++1}.$$Also, put $\delta_j:=\sqrt{\varepsilon_j}$. By the Ekeland Variational Principle, there is a sequence $v_j \in \overline{U}$ such that
	\begin{equation}\label{sequence vj}
	[u_{j,+}-v_j]_{s, G} \leq \delta_j, \quad \Phi(v_j)\leq \Phi(u_{j,+})<m+\varepsilon_j
\end{equation}and
\begin{equation}\label{sequence vj 2}
\Phi(v_j)< \Phi(w)+\dfrac{\varepsilon_j}{\delta_j}[v_j-w]_{s, G}, \quad w \in \overline{U}, w \neq v_j.
\end{equation}We will prove next that $v_j \in U$. To get a contradiction, assume that $[v_{j, -}]_{s, G} = c$. Then,
$$\Phi(-v_{j, -}) \geq c[v_{j, -}]_{s, G}^{p^+}=c^{p^++1} > \varepsilon_j.$$On the other hand, since $v_{j, +}\in K_1$, we have from \eqref{ineq u pm}
$$\Phi(v_j)> \Phi(v_{j, +})+ \Phi(-v_{j, -}) \geq m+ \varepsilon_j,$$which contradicts \eqref{sequence vj}. Thus, 	 $v_j \in U$.
Next, since  $v_j \in U$ with $U$ open,  by \eqref{sequence vj 2}, there holds $(\Phi_{\vert_{M_1}})'(v_j) \to 0$. By Lemma \ref{L28}, $v_j$ contains a convergence subsequence, also denoted $v_j$, with limit $v$. From \eqref{sequence vj}, we get $u_{j, +} \to v, $ and from the completeness of $K_1$, it follows $v \in K_1$ and so $v \in U$ ($v$ is in the interior of $M_1$). Moreover, by continuity, $(\Phi_{|_{M_1}})'(v)=0$. By Lemma \ref{L29}, $v$ is a critical point of $\Phi$ as well. Similarly, we can state the result for $K_2$ and $K_3$. 
	\end{proof}
\section*{Conflict of interest}
On behalf of all authors, the corresponding author states that there is no conflict of interest.
\section*{Acknowledgements}
A.S. and M.J.S.M are partially supported by ANPCyT under grants PICT 2017-0704, PICT 2019-3837 and by Universidad Nacional de San Luis under grants PROIPRO 03-2420. P. O. is supported by Proyecto Bienal  B080 Tipo 1 (Res. 4142/2019-R). The authors thank the referee for her/his useful comments to improve the manuscript.

\end{document}